\documentclass[11pt, reqno, oneside, notitlepage]{amsart}

\usepackage[a4paper, total={6in, 9.1in}]{geometry}

\usepackage{amsmath,amscd}
\usepackage{amssymb}
\usepackage{amsthm}
\usepackage{comment}
\usepackage{graphicx}
\usepackage{epstopdf} 
\usepackage{mathrsfs}
\usepackage{cite}

\usepackage{bm}
\usepackage{hyperref}
\usepackage{xcolor}
\hypersetup{
	colorlinks,
	linkcolor={blue},
	urlcolor={blue},
	citecolor={red}
}

\theoremstyle{plain}
\newtheorem{thm}{Theorem}[section]
\newtheorem{prop}{Proposition}[section]

\newtheorem{lem}[prop]{Lemma}

\newtheorem{defi}[prop]{Definition}
\newtheorem{rmk}[prop]{Remark}

\newtheorem*{proposition*}{Proposition}

\numberwithin{equation}{section}

\newcommand {\p} {\partial}

\title[Inverse Problems for Mean Field Games]{Inverse Problems for Mean Field Games}

\author[H. Liu]{Hongyu Liu}
\address{Department of Mathematics, City University of Hong Kong, Kowloon, Hong Kong SAR, China}
\email{hongyu.liuip@gmail.com, hongyliu@cityu.edu.hk}

\author[C. Mou]{Chenchen Mou}
\address{Department of Mathematics, City University of Hong Kong, Kowloon, Hong Kong SAR, China}
\email{chencmou@cityu.edu.hk}

\author[S. Zhang]{Shen Zhang}
\address{Department of Mathematics, City University of Hong Kong, Kowloon, Hong Kong SAR, China}
\email{szhang347-c@my.cityu.edu.hk}

\begin{document}
	\maketitle
	
	\begin{abstract}
		
		The theory of mean field games studies the limiting behaviors of large systems where the agents interact with each other in a certain symmetric way. The running and terminal costs are critical for the agents to decide the strategies. However, in practice they are often partially known or totally unknown for the agents, while the total cost is known at the end of the game. To address this challenging issue,  we propose and study several inverse problems for mean field games. When the Lagrangian is a kinetic energy, we first establish unique identifiability results, showing that one can recover either the running cost or the terminal cost from knowledge of the total cost. If the running cost is limited to the time-independent class, we can further prove that one can simultaneously recover both the running and the terminal costs. Finally, we extend the results to the setup with general Lagrangians.
	\end{abstract}

	\tableofcontents
	\section{Introduction}
	The theory of mean field games (MFGs) was introduced and studied by Caines-Huang-Malham\'e \cite{HCM06,HCM071,HCM072,HCM073} and Lasry-Lions \cite{LL06a, LL06b, LL07a, Lions} independently in 2006.  The MFG theory has rapidly developed into one of the most significant tools towards the study of the Nash equilibrium behavior of large systems. Such problems consider limit behavior of large systems where the homogeneous strategic players interact with each other in a certain symmetric way. More precisely, each player acts according to his/her optimization problem taking into account other players' decisions. Since their population is large, we can assume the number of players goes to infinity and hence a representative player exists. They have a wide variety of applications, including economics \cite{AchdouHanLasryLionsMoll}, engineering \cite{HCM06}, finance \cite{LackerZari}, social science \cite{Gelfand} and many others. We refer to Lions \cite{Lions}, Cardaliaguet \cite{Cardaliaguet} and Bensoussan-Frehse-Yam \cite{BFY} for introductions of the subject in its early stage and Carmona-Delarue \cite{CarDel-I, CarDel-II} and Cardaliaguet-Porretta \cite{CarPor} for comprehensive accounts on the state-of-the-art developments in the literature. 
	
	\smallskip
	
	We first briefly introduce the mathematical setup of our study and shall supplement more details in Section 2. In its typical formulation, an MFG can be described as follows. Let $n\in\mathbb{N}$ and the quotient space $\mathbb{T}^n:=\mathbb{R}^n\backslash\ \mathbb{Z}^n$ be the $n$-dimensional torus, which signifies a state space. Given $x\in\mathbb T^n$ and the flow of probability measures $\{\rho_t\}_{t\in [0,T]}$ on $\mathbb T^n$ with $\rho_0=m_0$, one aims at minimizing the cost functional over all the admissible closed-loop controls:
	\begin{equation}\label{eq:mfgproblem}
		J(x;\{\rho_t\}_{t\in[0,T]},\alpha)=\inf_{\alpha}\mathbb E\left\{\int_0^T L(X_t^{x,\alpha},\alpha(t, X_t^{x,\alpha}))+F(X_t^{x,\alpha},\rho_t)dt+G(X_T,\rho_T)\right\},
	\end{equation}
	such that
	\begin{equation}\label{eq:mfgconstrain}
		X_t^{x,\alpha}=x+\int_0^t\alpha(s,X_s^{x,\alpha})ds+\sqrt{2} B_s+\mathbb Z^n\quad\text{on $[0,T]$,}
	\end{equation}
	where $L:\mathbb T^n\times\mathbb R^n\to\mathbb R$ is a Lagrangian, $F:\mathbb T^n\times\mathcal{P}(\mathbb T^n)\to\mathbb R$ is a running cost and $G:\mathbb T^n\times\mathcal{P}(\mathbb T^n)\to\mathbb R$ is a terminal cost. We call $(\alpha^*,\{\rho_{t}^*\}_{t\in[0,T]})$ a mean field equilibrium if
	\[
	\rho_0^*=m_0\quad\text{and}\quad\alpha^*:=\arg\min_{\alpha}J(x;\{\rho_t^*\}_{t\in[0,T]},\alpha),
	\]
	and the law of $X_t^{\xi,\alpha^*}$ on $\mathbb T^n$ is $\rho_t^*$ where 
	\begin{equation}\label{eq:mfgpopulation}
		X_t^{\xi,\alpha^*}=\xi_0+\int_0^{t}\alpha^*(s,X_s^{x,\alpha^*})ds+\sqrt{2}B_t+\mathbb Z^n\quad \text{on [0,T]},
	\end{equation}
	and its initial status $\xi_0$ is a random variable with the law $m_0$ on $\mathbb T^n$. The mean field equilibrium can be characterized by the following MFG system:
	\begin{equation}\label{eq:mfg}
		\left\{
		\begin{array}{ll}
			-\partial_t u(x,t) -\Delta u(t,x)+ H\big(x,\nabla u(x,t)\big)-F(x,t,m(x,t))=0,&  {\rm{in}}\ \mathbb T^n\times (0,T),\medskip\\
			\partial_tm(x,t)-\Delta m(x,t)-{\rm div} \big(m(x,t) \nabla_pH(x, \nabla u(x,t)\big)=0, & {\rm{in}}\ \mathbb T^n\times(0,T),\medskip\\
			u(x,T)=G(x,m(x,T)),\ m(x,0)=m_0(x), & {\rm{in}}\ \mathbb T^n.
		\end{array}
		\right.
	\end{equation}
	In \eqref{eq:mfg}, $\Delta$ and $\rm{div}$ are the Laplacian and divergent operators with respect to the $x$-variable, respectively. The Hamiltonian $H$ is the Legendre-Fenchel transform of the Lagrangian $L$ in \eqref{eq:mfgproblem}. Here, $H(x, \nabla u)=H(x, p)$ with $(x, p):=(x, \nabla u)\in \mathbb{T}^n\times\mathbb{R}^{n}$ being the canonical coordinates. In the physical setup, $u$ is the value function of each player; $m$ signifies the population distribution; $F$ is the running cost function which signifies the interaction between the agents and the population; $m_0$ represents the initial population distribution and $G$ signifies the terminal cost. All the functions involved are real valued and periodically extended from $\mathbb{T}^n$ to $\mathbb{R}^n$, which means that we are mainly concerned with periodic boundary conditions for the MFG system \eqref{eq:mfg}.  In particular, we note that $m(\cdot, t)$ is required to be a probability measure for any $t\in [0, T]$. That is, it is required that for any given $t\in [0, T]$:
		\begin{equation}\label{eq:pm1}
			m(x, t)=m_t(x)\in\mathcal{O}_a:=\{\mathfrak{m}(x):\mathbb{T}^n\to[0,\infty)\,\Big|\,\int_{\mathbb{T}^n}\,\mathfrak{m}\,dx=a\leq1 \}.
		\end{equation}
		Here, we point out that by applying the divergence theorem to the second equation in \eqref{eq:mfg}, one can directly verify that if $\int_{\mathbb{T}^n} m_0(x)\, dx=a$, then $\int_{\mathbb{T}^n} m(x, t)\, dx=a$ for any subsequent $t\in (0, T]$. However, the non-negativity of $m_0$ and $m$ should be imposed in order to guarantee that they are probability measures. In principle, one would also need to require that $a=1$ which signifies that the game agents are confined within a given domain.
		Throughout the current study, we consider a specific scenario that the MFG domain consists of a family of disjoint subdomains, say $\Sigma_j$, $j\in\mathbb{N}$, such that the overall population on $\cup_j \Sigma_j$ is 1, namely $\int_{\cup_j \Sigma_j} m=1$. Though those subdomains are disjoint, the agents within each subdomain can interact with those in other subdomains, say e.g. via internet. Hence if $\mathbb{T}^n$ is taken to be any one of those subdomains, i.e. $\Sigma_j$, it is not necessary to require that $\int_{\mathbb{T}^n} m=1$. That is, $a$ in \eqref{eq:pm1} can be any number in $[0, 1]$, as long as $m$ is required to be nonnegative. This technical relaxation is crucial in our subsequent study but practically unobjectionable. In this setup, the mean field strategy can be formally represented by $\alpha^*=-\nabla_pH(x,\nabla u(x,t))$. In Section~2 in what follows, we shall supplement more background introduction on the MFG system.

	\smallskip
	
	The well-posedness of the MFG system \eqref{eq:mfg} is well-understood in various settings. The first results date back to the original works of Lasry and Lions and have been presented in Lions \cite{Lions} and see also Caines-Huang-Malhame \cite{HCM06}. Many progresses have been made afterwards. Regarding $F$ and $G$, one can consider both non-local and local dependences on the measure $m$. The well-posedness of the MFG system \eqref{eq:mfg} is known in Cardaliaguet \cite{Cardaliaguet}, Cardaliaguet-Porretta \cite{CarPor}, Carmona-Delarue \cite{CarDel-I}, Meszaros-Mou \cite{MM} in the case of nonlocal data $F$ and $G$; and Ambrose\cite{Amb:18, Amb:21}, Cardaliaguet\cite{Car}, Cardaliaguet-Graber\cite{CarGra}, Cardaliaguet-Graber-Porretta-Tonon\cite{CarGraPorTon}, Cardaliaguet-Porretta \cite{CarPor}, Cirant-Gianni-Mannucci\cite{CirGiaMan}, Cirant-Goffi\cite{CirGof}, Ferreira-Gomez\cite{FerGom}, Ferreira-Gomez-Tada\cite{FerGomTad}, Gomez-Pimentel-Sanchez Morgado\cite{GomPimSan:15,GomPimSan:16}, Porretta \cite{Por} in the case that $F,G$ are locally dependent on the measure variable $m$. 
	
	\smallskip
	
	We term the above well-posed MFG system \eqref{eq:mfg} to be the forward problem. In this paper, we are mainly concerned with the inverse problem of determining the running cost $F$ or the terminal cost $G$ by knowledge of the total cost associated with the above MFG system. To that end, we introduce a measurement map $\mathcal{M}_{F,G}$ as follows:
	\begin{equation}\label{eq:M}
		\mathcal{M}_{F, G}(m_0(x))=u(x,t)\big|_{t=0},\quad x\in\mathbb{T}^n,
	\end{equation}
	where $\color{blue} m_0(x)\in\mathcal{O}_a$ and $u(x, t)$ are given in the MFG system \eqref{eq:mfg}. That is, for a given pair of $F$ and $G$, $\mathcal{M}_{F, G}$ sends a prescribed initial population distribution $m_0$ to $u(x, 0)$, which signifies the total cost of the MFG \eqref{eq:mfg}. In Section 3, we shall show that $\mathcal{M}_{F,G}$ is well-defined in proper function spaces. The inverse problem mentioned above can be formulated as:
	\begin{equation}\label{eq:ip1}
		\mathcal{M}_{F,G}\longrightarrow F\ \mbox{or/and}\ G. 
	\end{equation}
	In the mean field game theory, the running cost $F$ and the terminal cost $G$ are critical for the agents to decide the strategies. However, in practice they are often partially known or totally unknown for the agents, while the total cost $u(\cdot, 0)$ can be measured at the end of the game. This is a major motivation for us to propose and study the inverse problem \eqref{eq:ip1}. We believe our study could have many applications in the areas mentioned above. An example in our mind is the produce pricing. Suppose that in the market there are many companies are producing the same product for selling to make profits. As a customer, we do not have the information on the precise production cost, however we do know the selling price of the product at the end. Therefore, the recovery of the production cost is a typical inverse problem in the mean field game.

	In this paper, we are mainly concerned with the unique identifiability issue, which is of primary importance for a generic inverse problem. In its general formulation, the unique identifiability asks whether one can establish the following one-to-one correspondence:
	\begin{equation}\label{eq:ip2}
		\mathcal{M}_{F_1, G_1}=\mathcal{M}_{F_2, G_2}\quad\mbox{if and only if}\quad (F_1, G_1)=(F_2, G_2), 
	\end{equation}
	where $(F_j, G_j)$, $j=1,2$, are two configurations. 
	
	\smallskip
	
	Unlike the forward problem of MFGs, the theory of the inverse problem has not yet been well-established. To the best of our knowledge, only some numerical studies have been conducted to the inverse problem of MFGs. It starts from the recent work Ding-Li-Osher-Yin \cite{DingLiOsherYin}. The authors reconstructed the running cost from the observation of the distribution of the population and the agents' strategy. The running cost consists of a kinetic energy (with an unknown underlining metric) and a convolution-type running cost. The main goal there is to numerically recover the underlining metric and the convolution kernel. Another numerical work Chow-Fung-Liu-Nurbekyan-Osher \cite{ChowFungLiuNurbekyanOsher} considered a different inverse problem of MFGs. The work focused on the recovery of the running cost from a finite number of the boundary measurements of the population profile and boundary movement. Both studies mentioned above consider the MFG model that the running cost is non-locally dependent on the measure variable $m$. We would also like to mention a related study in \cite{initial vector}, where the authors proved that if additional knowledge about the initial vector $Du(x,0)$ is given, then the solutions to the MFG system are unique.
	
	\smallskip

	In our study of the inverse problem \eqref{eq:ip1}, we are mainly concerned with the data locally depending on the measure variable, i.e. $F(x, t, m(\cdot, t)):=F(x, t, m(x, t))$ and $G(x,m(\cdot, T)):=G(x,m(x, T))$. The model is motivated from the traffic flow and the crowd motion problems. For the problems, the cost depends only on the distribution of the population locally. We assume all the agents are rational and the observer only knows the total cost of agents at the end. The main goal is to recover the running or/and terminal costs. Let us briefly introduce the main results we prove in the paper. When the Lagrangian is a kinetic energy, we first show thatthe terminal cost $G$ is uniquely identifiable by the measurement map $\mathcal{M}_{F,G}$ by assuming the running cost $F$ is a-priori known. We emphasize that, for this inverse problem, we assume that the running and the terminal costs satisfy $F(x, t, 0)=G(x,0)=0$ and we justify that the assumption is necessary for both unique identifiability problems. Moreover, the running cost is allowed to be time-dependent. If the running cost is limited to the time-independent class, we further show that we can recover both the running and the terminal costs with the given measurement map $\mathcal{M}_{F,G}$. Finally, we extend a large extent of the above unique identifiability results to general Lagrangians. To establish those theoretical unique identifiability results, we develop novel mathematical strategies that make full use of the intrinsic structure of the MFG system. Our study opens up a new field of research on inverse problems for mean field games with many potential developments.

	\medskip
	
	The rest of the paper is organized as follows. We introduce the admissibility assumptions on $F$ and $G$ and state the main results of this paper in Section 2. In Section 3, we establish certain well-posedness results of the forward MFG system, which shall be needed for the inverse problems. We discuss the admissibility assumptions in Section 4. By counter examples, we show that those assumptions are unobjectionable for the inverse problems. Finally, we show various unique identifiability results in Section 5 and some generalizations in Section 6.
	
	\section{Preliminaries and Statement of Main Results}
	\subsection{Notations and Basic Setting}
	As introduced earlier, we let $n\in\mathbb{N}$ and $\mathbb{T}^n:=\mathbb{R}^n\backslash\ \mathbb{Z}^n$ be the $n$-dimensional torus. Set $x=(x_1,x_2,\ldots, x_n)\in\mathbb{R}^n$. 
	If $f(x): x\in\mathbb{T}^n\to \mathbb{R}$ is smooth and $l=(l_1,l_2,...,l_n)\in\mathbb{N}_0^n$ is a multi-index with $\mathbb{N}_0:=\mathbb{N}\cup\{0\}$, then $D^lf$ stands for the
	derivative $\frac{\p ^{l_1}}{\p x_1^{l_1}}...\frac{\p ^{l_n}}{\p x_n^{l_n}}f$. Given $\nu\in\mathbb{S}^{n-1}:=\{x\in\mathbb{R}^n; |x|=1\}$,
	we also denote by $\p_{\nu}f$ the directional derivative of $f$ in the direction $\nu$. For $k\in\mathbb{N}_0$ and $\alpha\in [0,1)$, we say $f\in C^{k+\alpha} (k\in\mathbb{N}_0)$ if $D^lf$ exists and $\alpha$- H\"older continuous for any $l\in\mathbb{N}_0^n$ with $|l|\leq k$. Define 
		\begin{equation}\label{eq:pf1}
			C_+^{k+\alpha}(\mathbb{T}^n):=\{f(x)\in C^{k+\alpha}(\mathbb{T}^n) : f(x)\geq0 \}.
		\end{equation}
		It is remarked that the set $C_+^{k+\alpha}(\mathbb{T}^n)$ shall be needed in order to fulfil the probability measure constraint in our subsequent analysis; see also \eqref{eq:pm1}.

	For functions $f:\mathbb T^n\times (0,T)\to\mathbb R$, we say $f$ belongs to $C^{k+\alpha,\frac{k+\alpha}{2}}$ if $D^l D_t^jf$ exists for any $l\in\mathbb{N}_0^n$ and $j\in\mathbb{N}_0$ with $ |l|+2j\leq k$ and 
	\[
	\sup_{(x_1,t_1),(x_2,t_2)\in \mathbb T^n\times (0,T)}\frac{|D^lD_t^jf(x_1,t_1)-D^lD_t^jf(x_2,t_2)|}{|x_1-x_2|^\alpha+|t_1-t_2|^{\frac{\alpha}{2}}}<\infty,
	\]	
	for any $l\in\mathbb N_0^n$ and $j\in\mathbb N_0$ with $|l|+2j=k$.
	
	Throughout the paper, for a function $f$ define on $\mathbb{T}^n$ or $\mathbb{T}^n\times(0,T)$, it means that it is a periodic-$1$ function with respect to the space variable $x_j$, $1\leq j\leq N$. That is, it is a periodic-$(1,1,\ldots,1)$ function with respect to $x\in \mathbb{R}^n$. 
	
	\subsection{Mean Field Game}

	Let $\mathcal{P}(\mathbb T^n)$ and $\mathcal{P}(\mathbb R^n)$ denote the set of probability measures on $\mathbb T^n$ and $\mathbb R^n$ respectively. Let $(\Omega,\mathscr{F},\mathbb F,\mathbb P)$ be a filtered probability space; $B$ be an $\mathbb F$-adapted Brownian motion on $\mathbb R^n$; and we assume $\mathscr{F}_0$ is rich enough to support $\mathcal{P}(\mathbb T^n)$. For any $\mathscr{F}$-measurable random variable $\xi$, we denote the law of $\xi$ on $\mathbb R^n$ by $\mathcal{L}_\xi\in\mathcal{P}(\mathbb R^n)$ and the law of $\xi$ on $\mathbb T^n$ by $\mathcal{L}_{\xi+\mathbb Z^n}\in\mathcal{P}(\mathbb T^n)$. Moreover, for any sub-$\sigma$-algebra $\mathcal{G}\subset \mathscr{F}$ and any $m\in\mathcal P(\mathbb T^n)$, $\mathbb M(\mathcal{G};m)$ denotes the set of $\mathcal{G}$-measurable random variables $\xi$ on $\mathbb R^n$ such that $\mathcal{L}_{\xi+\mathbb Z^n}=m$.

	Our mean field game depends on the following data:
	\[
	L:\mathbb T^n\times\mathbb R^n\to\mathbb R,\quad F:\mathbb T^n\times\mathcal{P}(\mathbb T^n)\to\mathbb R\quad\text{and}\quad G:\mathbb T^n\times\mathcal{P}(\mathbb T^n)\to\mathbb R.
	\]
	Let $T>0$.  For any $t_0\in [0,T]$, we let $\mathscr{A}_{t_0}$ denote the set of admissible controls $\alpha:[t_0,T]\times\mathbb T^n\to\mathbb R^d$ which are Borel measurable, and uniformly Lipschitz continuous in $x$. We also denote $B_t^{t_0}:=B_t-B_{t_0}$, $B_t^{0,t_0}:=B_t^0-B_{t_0}^0$, $t\in[t_0,T]$.

	Given $x\in\mathbb T^n$, $\alpha\in\mathscr{A}_{t_0}$ and the flow of probability measures $\{\rho_t\}_{t\in[0,T]}\subset \mathcal{P}(\mathbb T^n)$ with $\rho_0=m_0$,
	the state of an agent satisfies the following controlled SDE (stochastic differential equation) on $[t_0,T]$:
	\begin{equation}\label{eq:ind}
		X_t^{t_0,x,\alpha}=x+\int_{t_0}^t\alpha(s,X_s^{t_0,x,\alpha})ds+\sqrt{2}B_t^{t_0}+\mathbb Z^n.
	\end{equation}
	Consider the conditionally expected cost for the mean field game:
	\begin{eqnarray}\label{eq:cost}
		&&J(t_0,x;\{\rho_{t}\}_{t\in [0,T]},\alpha):=\inf_{\alpha\in\mathscr{A}_{t_0}}\mathbb E\Big[\int_{t_0}^TL(X_t^{t_0,x,\alpha},\alpha(t,X_t^{t_0,x,\alpha}))+F(X_t^{t_0,x,\alpha},\rho_t)dt\nonumber\\
		&&\qquad\qquad\qquad\qquad\qquad\qquad+G(X_T^{t_0,x,\alpha},\rho_T)\Big].
	\end{eqnarray}
	\begin{defi}
		We say that $(\alpha^*,\{\rho_t^{*}\}_{t\in[0,T]})$ is a mean field equilibrium (MFE) if it satisfies the following properties:\\
		(i) $\rho_0^*=m_0$;\\
		(ii) for any $\xi_0\in \mathbb{M}(\mathcal{F}_0,m_0)$, we have $\mathcal{L}_{X_t^{0,\xi_0,\alpha^*}}=\rho_t^{*}$ where
		\[
		X_t^{0,\xi_0,\alpha^*}=\xi_0+\int_{0}^t\alpha^*(s,X_s^{0,\xi_0,\alpha^*})ds +\sqrt{2}B_t+\mathbb Z^n;
		\]
		(iii) for any $(t_0,x)\in[0,T]\times\mathbb T^n$, we have
		\begin{equation*}
			J(t_0,x;\{\rho_t^*\}_{t\in [0,T]},\alpha^*)=\inf_{\alpha\in\mathscr{A}_{t_0}}J(t_0,x;\{\rho_t^*\}_{t\in [0,T]},\alpha),\quad \text{for $\rho^*_{t_0}$-a.e. $x\in\mathbb T^n$.} 
		\end{equation*}
	\end{defi}
	When there is a unique MFE $(\alpha^*,\{\rho_t^{*}\}_{t\in[0,T]})$, then the mean field game leads to the following value function of the agent:
	\[
	u(t_0,x):=J(t_0,x;\{\rho_t^*\}_{t\in [0,T]},\alpha^*).
	\]
	Let $m(\cdot,t_0)=\rho_{t_0}^*$. Then $(u,m)$ solves the following mean field game system (cf. \cite{CarPor,Lions}):  
	\begin{equation}\label{main_equation}
		\begin{cases}
			\displaystyle{-\p_tu(x,t)-\Delta u(x,t)+\frac 1 2 |\nabla u(x, t)|^2= F(x,t,m(x,t)),}& \text{ in } \mathbb{T}^n\times(0,T),\smallskip\\
			\p_t m(x,t)-\Delta m(x,t)-{\rm div}\big(m(x,t)\nabla u(x,t)\big)=0,&\text{ in }\mathbb{T}^n\times(0,T),\smallskip \\
			u(x,T)=G(x,m(T,x)), \quad m(x,0)=m_0(x) &  \text{ in } \mathbb{T}^n,
		\end{cases}  	
	\end{equation}
	where as mentioned earlier, periodic boundary conditions are imposed on $\partial \mathbb{T}^n$ for $u$ and $m$. 
	
	\subsection{Inverse Problems}
	We recall the probability measure constraint $\mathcal{O}_a$ introduced in \eqref{eq:pm1}.
	Define the set
	\[
	\begin{split}
		\mathcal{E}_{F,G}:=& \{m_0\in C^{2+\alpha}(\mathbb{T}^n)\cap\mathcal{O}_a  : \text{the system }\eqref{main_equation}\\
		& \text{ has a unique solution in the sense described in Section 3 in what follows } \}.
	\end{split}
	\] 
	We introduce the following measurement map $\mathcal{M}_{F,G}$:
	\begin{align}\label{eq:G}
		\begin{split}
			\mathcal{M}_{F,G}: \mathcal{E}_{F,G} & \rightarrow  L^2(\mathbb{T}^n) , \\  
			m_0&\mapsto  \Big(x\in\mathbb{T}^n \mapsto u(x,t) \Big|_{t=0}\Big), 
		\end{split}
	\end{align}
	where $u(x,t)$ is the solution of $\eqref{main_equation}$ with initial data $m(x,0)=m_0(x).$
	
	In the first setup of our study, we consider the case that  $F$ and $G$ belong
	to an analytic class. Henceforth, we set
	\begin{equation}\label{eq:Q}
		Q=\overline{\mathbb{T}^n\times(0,T) },
	\end{equation}
	be the closure of $\mathbb{T}^n\times(0,T).$
	\begin{defi}\label{Admissible class1}
		We say $U(x,t,z):\mathbb{T}^n\times \mathbb{R}\times\mathbb{C}\to\mathbb{C}$ is admissible, denoted by $U\in \mathcal{A}$, if it satisfies the following conditions:
		\begin{enumerate}
			\item[(i)]~The map $z\mapsto U(\cdot,\cdot,z)$ is holomorphic with value in $C^{2+\alpha,1+\frac{\alpha}{2}}(Q)$ for some $\alpha\in(0,1)$;
			\item[(ii)] $U(x,t,0)=0$ for all $(x,t)\in\mathbb{T}^n\times (0,T).$ 	
		\end{enumerate}
		
		Clearly, if (1) and (2) are fulfilled, then $U$ can be expanded into a power series as follows:
		\begin{equation}\label{eq:F}
			U(x,t,z)=\sum_{k=1}^{\infty} U^{(k)}(x,t)\frac{z^k}{k!},
		\end{equation}
		where $ U^{(k)}(x,t)=\frac{\p^k U}{\p z^k}(x,t,0)\in C^{2+\alpha,1+\frac{\alpha}{2}}(Q).$
	\end{defi}
	%
	%
	%
	%
	
	\begin{defi}\label{Admissible class2}
		We say $U(x,z):\mathbb{T}^n\times\mathbb{C}\to\mathbb{C}$ is admissible, denoted by $U\in\mathcal{B}$, if it satisfies the following conditions:
		\begin{enumerate}
			\item[(i)] The map $z\mapsto U(\cdot,z)$ is holomorphic with value in $C^{2+\alpha}(\mathbb{T}^n)$ for some $\alpha\in(0,1)$;
			\item[(ii)] $U(x,0)=0$ for all $x\in\mathbb{T}^n.$
		\end{enumerate}

		Clearly, if (1) and (2) are fulfilled, then $U$ can be expanded into a power series as follows:
		\begin{equation}\label{eq:G}
			U(x,z)=\sum_{k=1}^{\infty} U^{(k)}(x)\frac{z^k}{k!},
		\end{equation}
		where $ U^{(k)}(x)=\frac{\p^kU}{\p z^k}(x,0)\in C^{2+\alpha}(\mathbb{T}^n).$
	\end{defi}
	
	\begin{rmk}\label{rem:1}
		The admissibility conditions in Definitions~\ref{Admissible class1} and \ref{Admissible class2} shall be imposed as a-priori conditions on the unknowns $F$ and $G$ in what follows for our inverse problem study. It is remarked that as noted earlier that both $F$ and $G$ are functions of real variables. However, for technical reasons, we extend the functions to the complex plane with respect to the $z$-variable, namely $U(\cdot,z)$ and $ U(\cdot,\cdot,z)$, and assume that they are holomorphic as functions of the complex variable $z$. This also means that we shall assume $F$ and $G$ are restrictions of those holomorphic functions to the real line. This technical assumption shall be used to show the well-posedness of the MFG system in section $\ref{section wp}.$ Throughout the paper, we also assume that in the series expansions \eqref{eq:F} and \eqref{eq:G}, the coefficient functions $U^{(k)}$ are real-valued. 
	\end{rmk}
	
	\begin{rmk}
		We would like to emphasise that the zero conditions, namely the admissibility conditions (ii) in Definitions~\ref{Admissible class1} and \ref{Admissible class2}, are unobjectionable to our inverse problem study. In fact, in Section~4 in what follows, we shall construct several MFG examples where the zero admissibility conditions are violated and the associated inverse problems have no unique identifiability results. 
	\end{rmk}
	
	We are in a position to state the first unique recovery result for the inverse problem \eqref{eq:ip1}, which shows that one can recover the terminal cost $G$ from the measurement map $\mathcal{M}$. Here and also in what follows, we sometimes drop the dependence on $F, G$ of $\mathcal{M}$, and in particular in the case that one quantity is a-priori known, say  $\mathcal{M}_F$ or $\mathcal{M}_G$, which should be clear from the context. 

	\begin{thm}\label{der g}
		Assume $F \in\mathcal{A}$, $G_j\in\mathcal{B}$ ($j=1,2$). Let $\mathcal{M}_{G_j}$ be the measurement map associated to
		the following system:
		\begin{equation}\label{eq:mfg2}
			\begin{cases}
				-\p_tu(x,t)-\Delta u(x,t)+\frac 1 2 {|\nabla u(x,t)|^2}= F(x,t,m(x,t)),& \text{ in } \mathbb{T}^n\times (0,T),\medskip\\
				\p_t m(x,t)-\Delta m(x,t)-{\rm div}(m(x,t)\nabla u(x,t))=0,&\text{ in }\mathbb{T}^n\times(0,T),\medskip\\
				u(x,T)=G_j(x,m(x,T)), & \text{ in } \mathbb{T}^n,\medskip\\
				m(x,0)=m_0(x), & \text{ in } \mathbb{T}^n.\\
			\end{cases}  		
		\end{equation}	
		If for any $m_0\in C^{2+\alpha}(\mathbb{T}^n)\cap\mathcal{O}_a$, one has 
		$$\mathcal{M}_{G_1}(m_0)=\mathcal{M}_{G_2}(m_0),$$    then it holds that		
		$$G_1(x,z)=G_2(x,z)\ \text{  in  } \mathbb{T}^n\times \mathbb{R}.$$ 
	\end{thm}
	Notice that in Theorems \ref{der g} we allow $F$ to depend on time. If we assume  $F$ depends only on $x$ and $m(x,t)$, we can determine $F$ and $G$ simultaneously. 
	
	\begin{thm}\label{der F,g}
		Assume $F_j,G_j \in\mathcal{B}$ ($j=1,2$) . Let $\mathcal{M}_{F_j,G_j}$ be the measurement map associated to
		the following system:
		\begin{equation}\label{eq:mfg3}
			\begin{cases}
				-\p_tu(x,t)-\Delta u(x,t)+\frac 1 2 {|\nabla u(x,t)|^2}= F_j(x,m(x,t)),& \text{ in }\mathbb{T}^n\times(0,T),\medskip\\
				\p_t m(x,t)-\Delta m(x,t)-{\rm div}(m(x,t)\nabla u(x,t))=0,&\text{ in }\mathbb{T}^n\times (0,T),\medskip\\
				u(x,T)=G_j(x,m(x,T)), & \text{ in } \mathbb{T}^n,\medskip\\
				m(x,0)=m_0(x), & \text{ in } \mathbb{T}^n.\\
			\end{cases}  		
		\end{equation}	
		If for any $m_0\in C^{2+\alpha}(\mathbb{T}^n)\cap\mathcal{O}_a$, one has 
		$$\mathcal{M}_{F_1,G_1}(m_0)=\mathcal{M}_{F_2,G_2}(m_0),$$    then it holds that	
		$$(G_1(x,z),F_1(x,z))=(G_2(x,z),F_2(x,z)) \ \text{  in  } \mathbb{T}^n\times \mathbb{R}.$$ 
	\end{thm}
	
	In Theorems~ \ref{der g} and \ref{der F,g}, the Lagrangian is of a quadratic form, namely $H(x,\nabla u)$ in \eqref{eq:mfg} is of the form $\frac 1 2 |\nabla u|^2$ (see \eqref{eq:mfg2}--\eqref{eq:mfg3}). In fact, we can extend a large extent of the results in those theorems to the case with a general Lagrangian. We choose to postpone the statement of those results in Section~6 along with their proofs. 
	
	
	\section{Well-posedness of the forward problems}\label{section wp}
	In this section, we show the well-posedness of the MFG systems in our study. The key point is the infinite differentiability of the equation with respect to a given (small) input $m_0(x).$ As  a preliminary, we recall the well-posedness result
	for linear parabolic equations \cite{Lady}\cite[Lemma 3.3]{ CarDelLasLio} .
	\begin{lem}\label{linear app unique}
		Consider the parabolic equation 
		\begin{equation}\label{linearapp wellpose}
			\begin{cases}
				-\p_tv(x,t)-\Delta v(x,t)+{\rm div} ( a(x,t)\cdot\nabla v(x,t))= f(x,t),& \text{ in }\mathbb{T}^n\times(0,T),\medskip\\
				v(x,0)=v_0(x), & \text{ in } \mathbb{T}^n, 
			\end{cases}  	
		\end{equation}
		where the periodic boundary condition is imposed on $v$. Suppose $a,f\in C^{\alpha,\frac{\alpha}{2}}(Q) $ and $v_0\in C^{2+\alpha}(\mathbb{T}^n)$, then $\eqref{linearapp wellpose}$ has a unique classical solution $v\in C^{2+\alpha,1+\frac{\alpha}{2}}(Q).$
	\end{lem}
	
	The following result is somewhat standard (especially Theorem \ref{local_wellpose}-(a)), while our technical conditions could be different from those in the literature. For completeness we provide a proof here. The idea is to differentiate the equation infinitely many times with respect to  the (small) input $m_0(x)$. We recall that $Q$ is defined in \eqref{eq:Q} and periodic boundary conditions are imposed to the MFG systems. 
	The following proof is based on the implicit functions theorem for Banach spaces. One may refer to \cite{Pos.J} for more related details about the theory of maps between Banach spaces. 
	\begin{thm}\label{local_wellpose}
		Suppose that $F\in\mathcal{A}$ and $G\in\mathcal{B}$. The following results holds:
		\begin{enumerate}
			
			\item[(a)]
			There exist constants $\delta>0$ and $C>0$ such that for any 
			\[
			m_0\in B_{\delta}(C^{2+\alpha}(\mathbb{T}^n)) :=\{m_0\in C^{2+\alpha}(\mathbb{T}^n): \|m_0\|_{C^{2+\alpha}(\mathbb{T}^n)}\leq\delta \},
			\]
			the MFG system $\eqref{main_equation}$ has a solution $u \in
			C^{2+\alpha,1+\frac{\alpha}{2}}(Q)$ which satisfies
			\begin{equation}\label{eq:nn1}
				\|(u,m)\|_{ C^{2+\alpha,1+\frac{\alpha}{2}}(Q)}:= \|u\|_{C^{2+\alpha,1+\frac{\alpha}{2}}(Q)}+ \|m\|_{C^{2+\alpha,1+\frac{\alpha}{2}}(Q)}\leq C\|m_0\|_{ C^{2+\alpha}(\mathbb{T}^n)}.
			\end{equation}
			Furthermore, the solution $(u,m)$ is unique within the class
			\begin{equation}\label{eq:nn2}
				\{ (u,m)\in  C^{2+\alpha,1+\frac{\alpha}{2}}(Q)\times C^{2+\alpha,1+\frac{\alpha}{2}}(Q): \|(u,m)\|_{ C^{2+\alpha,1+\frac{\alpha}{2}}(Q)}\leq C\delta \}.
			\end{equation}		
			\item[(b)] Define a function 
			\[
			S: B_{\delta}(C^{2+\alpha}(\mathbb{T}^n))\to C^{2+\alpha,1+\frac{\alpha}{2}}(Q)\times C^{2+\alpha,1+\frac{\alpha}{2}}(Q)\ \mbox{by $S(m_0):=(u,v)$}. 
			\] 
			where $(u,v)$ is the unique solution to the MFG system \eqref{main_equation}.
			Then for any $m_0\in B_{\delta}(C^{2+\alpha}(\mathbb T^n))$, $S$ is holomorphic.
		\end{enumerate}
	\end{thm}
	\begin{proof}
		Let 
		\begin{align*}
			&X_1:= C^{2+\alpha}(\mathbb{T}^n ), \\
			&X_2:=C^{2+\alpha,1+\frac{\alpha}{2}}(Q)\times C^{2+\alpha,1+\frac{\alpha}{2}}(Q),\\
			&X_3:=C^{2+\alpha}(\mathbb{T}^n)\times C^{2+\alpha}(\mathbb{T}^n)\times C^{\alpha,\frac{\alpha}{2}}(Q )\times C^{\alpha,\frac{\alpha}{2}}(Q ),
		\end{align*} and we define a map $\mathscr{K}:X_1\times X_2 \to X_3$ by that for any $(m_0,\tilde u,\tilde m)\in X_1\times X_2$,
		\begin{align*}
			&
			\mathscr{K}( m_0,\tilde u,\tilde m)(x,t)\\
			:=&\big( \tilde u(x,T)-G(x,\tilde m(x,T)), \tilde m(x,0)-m_0(x) , 
			-\p_t\tilde u(x,t)-\Delta \tilde u(x,t)\\ &+\frac{|\nabla \tilde u(x,t)|^2}{2}- F(x,t,\tilde m(x,t)), 
			\p_t \tilde m(x,t)-\Delta \tilde m(x,t)-{\rm div}(\tilde m(x,t)\nabla \tilde u(x,t))  \big) .
		\end{align*}

		First, we show that $\mathscr{K} $ is well-defined. Since the
		H\"older space is an algebra under the point-wise multiplication, we have $|\nabla u|^2, {\rm div}(m(x,t)\nabla u(x,t))  \in C^{\alpha,\frac{\alpha}{2}}(Q ).$
		By the Cauchy integral formula,
		\begin{equation}\label{eq:F1}
			F^{(k)}\leq \frac{k!}{R^k}\sup_{|z|=R}\|F(\cdot,\cdot,z)\|_{C^{\alpha,\frac{\alpha}{2}}(Q ) },\ \ R>0.
		\end{equation}
		Then there is $L>0$ such that for all $k\in\mathbb{N}$,
		\begin{equation}\label{eq:F2}
			\left\|\frac{F^{(k)}}{k!}m^k\right\|_{C^{\alpha,\frac{\alpha}{2}}(Q )}\leq \frac{L^k}{R^k}\|m\|^k_{C^{\alpha,\frac{\alpha}{2}}(Q )}\sup_{|z|=R}\|F(\cdot,\cdot,z)\|_{C^{\alpha,\frac{\alpha}{2}}(Q ) }.
		\end{equation}
		By choosing $R\in\mathbb{R}_+$ large enough and by virtue of \eqref{eq:F1} and \eqref{eq:F2}, it can be seen that the series \eqref{eq:F} converges in $C^{\alpha,\frac{\alpha}{2}}(Q )$ and therefore $F(x,m(x,t))\in  C^{\alpha,\frac{\alpha}{2}}(Q ).$ Similarly, we have $G(x,m(x,T))\in C^{2+\alpha}(\mathbb{T}^n).$ 
		This implies that $\mathscr{K} $ is well-defined.

		Let us show that $\mathscr{K}$ is holomorphic. Since $\mathscr{K}$ is clearly locally bounded, it suffices to verify that it is weakly holomorphic; see \cite[P.133 Theorem 1]{Pos.J}. That is we aim to show the map
		$$\lambda\in\mathbb C \mapsto \mathscr{K}((m_0,\tilde u,\tilde m)+\lambda (\bar m_0,\bar u,\bar m))\in X_3,\quad\text{for any $(\bar m_0,\bar u,\bar m)\in X_1\times X_2$}$$
		is holomorphic. In fact, this follows from the condition that $F\in\mathcal{A}$ and $G\in\mathcal{B}$.

		Note that $  \mathscr{K}(0,0,0)=0$. Let us compute $\nabla_{(\tilde u,\tilde m)} \mathscr{K} (0,0,0)$:
		\begin{equation}\label{Fer diff}
			\begin{aligned}
				\nabla_{(\tilde u,\tilde m)} \mathscr{K}(0,0,0) (u,m) =(& u|_{t=T}-G^{(1)}m(x,T), m|_{t=0}, \\
				&-\p_tu(x,t)-\Delta u(x,t)-F^{(1)}m, \p_t m(x,t)-\Delta m(x,t)).
			\end{aligned}			
		\end{equation}

		By Lemma $\ref{linear app unique}$, 
		if $\nabla_{(\tilde u,\tilde m)} \mathscr{K} (0,0,0)=0$, we have 
		$ \tilde m=0$ and then $\tilde u=0$. Therefore, the map is injective. 
		
		On the other hand,  letting $ (r(x),s(x,t))\in C^{2+\alpha}(\mathbb{T}^n)\times C^{\alpha,\frac{\alpha}{2}}(Q ) $,
		and by Lemma $\ref{linear app unique}$, there exists $a(x,t)\in C^{2+\alpha,1+\frac{\alpha}{2}}(Q)$ such that
		\begin{equation*}
			\begin{cases}
				\p_t a(x,t)-\Delta a(x,t)=s(x,t)  &\text{ in } \mathbb{T}^n,\medskip\\
				a(x,0)=r(x)                       &  \text{ in } \mathbb{T}^n .
			\end{cases}
		\end{equation*}
		Then letting $ (r'(x),s'(x,t))\in C^{2+\alpha}(\mathbb{T}^n)\times C^{\alpha,\frac{\alpha}{2}}(Q ) $, one can show that there exists $ b(x,t)\in C^{2+\alpha,1+\frac{\alpha}{2}}(Q)$ such that
		\begin{equation*}
			\begin{cases}
				-\p_t b(x,t)-\Delta b(x,t)-F^{(1)}a=s'(x,t)  &\text{ in } \mathbb{T}^n,\medskip\\
				b(x,T)=G^{(1)}a(x,T)+r'(x)                  &  \text{ in } \mathbb{T}^n.
			\end{cases}
		\end{equation*}
		
		Therefore, $\nabla_{(\tilde u,\tilde m)} \mathscr{K} (0,0,0)$ is a linear isomorphism between $X_2$ and $X_3$. Hence, by the implicit function theorem, there exist $\delta>0$ and a unique holomorphic function $S: B_{\delta}(\mathbb{T}^n)\to X_2$ such that $\mathscr{K}(m_0,S(m_0))=0$ for all $m_0\in B_{\delta}(\mathbb{T}^n) $.
		
		By letting $(u,m)=S(m_0)$, we obtain the unique solution of the MFG system \eqref{main_equation}. Let $ (u_0,v_0)=S(0)$.   Since $S$ is Lipschitz, we know that there exist constants $C,C'>0$ such that 
		\begin{equation*}
			\begin{aligned}
				&\|(u,m)\|_{ C^{2+\alpha,1+\frac{\alpha}{2}}(Q)^2}\\
				\leq& C'\|m_0\|_{B_{\delta}(\mathbb{T}^n)} +\|u_0\|_	{ C^{2+\alpha,1+\frac{\alpha}{2}}(Q)}+\|v_0\|_{ C^{2+\alpha,1+\frac{\alpha}{2}}(Q)}\\
				\leq& C \|m_0\|_{B_{\delta}(\mathbb{T}^n)}.
			\end{aligned}
		\end{equation*}
		
		The proof is complete. 
	\end{proof}
	\begin{rmk}
		Regarding the local well-posedness, several remarks are in order.
		\begin{enumerate}
			
			\item[(a)] The conditions on $F$ and $G$ (Definition \ref{Admissible class1}-(i) and $G$ satisfies Definition \ref{Admissible class2}-(i) ) are not essential and it is for convenience to
			apply implicit function theorem . Also, the analytic conditions on $F$ and $G$ can be replayed by weaker regularity conditions in the proof of the local well-posedness  \cite{Lions} , but these conditions will be utilized in our 
			inverse problem study. 
			
			\item[(b)] In order to apply the higher order linearization method that shall be developed in Section 5 for the inverse problems, we need the infinite differentiability of the equation with respect to the given input $m_0(x)$, it is shown by the fact that the solution map $S$ is holomorphic.

			\item[(c)] In the proof of Theorem $\ref{local_wellpose}$, we show the solution map $S$ is holomorphic. As a corollary,  the measurement map $\mathcal{M}=(\pi_1\circ S)\Big|_{t=0}$ is also holomorphic, where $\pi_1$ is the projection map with respect to the first variable.
			
		\end{enumerate}	
		
	\end{rmk}
	\section{Non-uniqueness and discussion on the zero admissibility conditions}
	In this section, we show that the zero admissibility conditions, namely $F(x,t, 0)=0$ and $G(x,0)=0$ in Definitions~\ref{Admissible class1} and \ref{Admissible class2} are unobjectionably necessary if one intends to uniquely recover $F$ or $G$ by knowledge of the measurement operator $\mathcal{M}_{F, G}$ for the inverse problem \eqref{eq:ip1}. For simplicity, we only consider the case that the space dimension $n=1$ without the periodic boundary conditions. That is, we consider the following MFG system:
	\begin{equation}\label{dim1}
		\begin{cases}
			-\p_tu_j(x,t)-\p_{xx} u_j(x,t)+\frac 1 2 {|\p_x u_j(x)|^2}= F_j(x,t,v_j(x,t)),& \text{ in } \mathbb{R}\times (0,T),\medskip\\
			\p_t v_j(x,t)-\p_{xx} v_j(x,t)-\p_x(v_j(x,t)\p_x u_j(x,t))=0,&\text{ in } \mathbb{R}\times(0,T),\medskip\\
			u_j(x,T)= G_j(x,v_j(x,T)), & \text{ in } \mathbb{R},\medskip\\
			v_j(x,0)=m_0(x), & \text{ in } \mathbb{R}.\\
		\end{cases}  		
	\end{equation}
	Furthermore, we assume $T$ is small enough such that the solution of the MFG system \eqref{dim1} is unique \cite{Amb:18,Amb:21,Amb22,Cira,Lions}. In what follows, we construct examples to show that if the zero admissibility conditions are violated then the corresponding inverse problems do not have uniqueness. 
	
	\begin{prop}
		Consider the system $\eqref{dim1}$. There exist $F_1=F_2\in C^{\infty}(\mathbb{R}\times\mathbb{R}\times\mathbb{R})$ and $G_1\neq G_2\in C^{\infty}(\mathbb{R}\times\mathbb{R})$ (but we do not have $G_1(x,0)=G_2(x,0)=0$) such that  the corresponding two systems admit the same measurement map, i.e. $\mathcal{M}_{G_1}=\mathcal{M}_{G_2}$.
	\end{prop}
	\begin{proof}
		Set 
		\[
		F_1=F_2=-\sin(x)+\frac{1}{4}(e^t-1)^2\cos^2(x),
		\] 
		and
		\[
		G_1=(e^T-1)\sin(x),\quad G_2=(1-e^T)\sin(x).
		\]
		It can be directly verified that 
		\[
		u_1(x,t)=(e^t-1)\sin(x)\quad\mbox{and}\quad u_2(x,t)=(1-e^t)\sin(x),
		\]  
		satisfy the corresponding  system. In this case, we have $\mathcal{M}_{G_1}(m_0)=\mathcal{M}_{G_2}(m_0)=0$ for any admissible $m_0$. 
	\end{proof}
	
	\begin{prop}\label{2}
		Consider the system $\eqref{dim1}$. There exist $G_1=G_2\in C^{\infty}(\mathbb{R}\times\mathbb{R})$ and $F_1\neq F_2\in C^{\infty}(\mathbb{R}\times\mathbb{R}\times\mathbb{R})$ (but we do not have $F_j(x,t,0)=0$, $j=1,2$) such that  the corresponding two systems admit the same measurement map, i.e. $\mathcal{M}_{F_1,G_1}=\mathcal{M}_{F_2,G_2}$.
	\end{prop}
	\begin{proof}
		Set
		\[
		F_1=-x(2t-T)+\frac{t^2(t-T)^2}{2},\quad F_2=-2x(2t-T)+2t^2(t-T)^2,
		\]
		and
		\[
		G_1=G_2=0. 
		\]
		Here, it is noted that $F_1$ and $F_2$ are independent of $v$. In such a case, it is straightforward to verify that $u_j(x,t)=jxt(t-T)$ is the solution of  the corresponding  system \eqref{dim1}. Clearly, one has $\mathcal{M}_{F_1}(m_0)=\mathcal{M}_{F_2}(m_0)=0$ for any admissible $m_0$. 
	\end{proof}
	
	Moreover, we can find $F_1, F_2\in C^{\infty}(\mathbb{R}\times\mathbb{R})$ which are independent of $t$ such that Proposition $\ref{2}$ holds.
	
	\begin{proof}

		Define
		$$Lu_j:=-\p_tu(x,t)-\p_{xx} u(x,t)+\frac{|\p_x u|^2}{2}.$$ 
		It is sufficient for us to show that there exist $u_1(x,t),u_2(x,t)$ such that
		\begin{enumerate}		
			\item[(1)] $L u_1\neq L u_2$ and $\p_t (Lu_j)=0 $ for $ j=1,2$;
			
			\item[(2)] $u_1(x,0)=u_2(x,0)$ and $u_1(x,T)=u_2(x,T)$.
		\end{enumerate}
		In fact, if this is true, we can set $F_j= Lu_j$ and $G(x)=u_1(x,T)$. Then one has $G_1=G_2.$
		
		Without loss of generality, we assume $T=1.$	
		Let $p(t)$ be a non-zero solution of the following ordinary differential equation (ODE):
		\begin{equation*}
			(	\ln( p'(t)))'=\frac{\sqrt{1+4t}}{2},
		\end{equation*}
		and $q(t)$ be a solution of the ODE:
		\begin{equation*}
			\begin{cases}
				&2q'(t)+\sqrt{1+4t}\, q''(t)=p(t)p'(t)\sqrt{1+4t},\medskip\\
				&q(0)=0.
			\end{cases}
		\end{equation*}
		With $p(t)$ and $q(t)$ given above, we can set 
		\[
		u_1(x,t)=p(t(t-1))x+q(t(t-1))\quad\mbox{and}\quad u_2(x,t)=q(t(t-1))x+2q(t(t-1)).
		\]
		It can be directly verified that $u_1$ and $u_2$ fulfil the requirements (1) and (2) stated above. 
	\end{proof}
	
	Finally, we would like to remark that by following a similar spirit, one may construct similar examples as those in Proposition 4.1 and 4.2 to the MFG system \eqref{dim1} associated with a periodic boundary condition. However, this shall involve a bit more tedious calculations and is not the focus of the current study. We choose not to explore further. As also stated earlier, it is unobjectionable to see that the zero admissibility conditions are necessary for the inverse problem study. 
	
	\section{Proofs of Theorems~\ref{der g} and \ref{der F,g}}
	
	In this section, we present the proofs of the three main theorems, namely Theorems, \ref{der g} and \ref{der F,g}. 
	To that end, we first introduce a higher order linearization procedure associated with the MFG system \eqref{main_equation} which shall be repeatedly used in the proofs. We also refer to \cite{LLLZ} where a higher order linearization procedure was considered for a semi-linear parabolic equation. 
	%
	%
	%
	
	Throughout the current section, if $f$ is a function defined on $\mathbb{T}^n$, we still use $f$ to denote the function obtained by extending $f$ to $\mathbb{R}^n$ periodically. 
	
	\subsection{Higher-order linearization}\label{HLM}
	This method depends on the infinite differentiability of the solution with respect to a given input $m_0(x)$, which was derived in Theorem~$\ref{local_wellpose}$.
	In fact, Cardaliaguet, Delarue, Lasry and Lions developed this linearization method in some probability measure space; see \cite{Book11}. However, the setup of our study is not completely covered by the discussion in \cite{Book11} and for completeness and self-containedness, we show the process in what follows.
	
	First, we introduce the basic setting of this higher order
	linearization method. Consider the system $\eqref{main_equation}$. Let 
	$$m_0(x;\varepsilon)=\sum_{l=1}^{N}\varepsilon_lf_l,$$
	where $f_l\in C_+^{2+\alpha}(\mathbb{T}^n)$ and $\varepsilon=(\varepsilon_1,\varepsilon_2,...,\varepsilon_N)\in\mathbb{R}_+^N$ ($\mathbb{R}_+:=\{x\in\mathbb{R}: x\geq0\}$) with 
		$|\varepsilon|=\sum_{l=1}^{N}|\varepsilon_l|$ small enough. Then $m_0\in C^{2+\alpha}(\mathbb{T}^n)\cap\mathcal{O}_a$. By Theorem $\ref{local_wellpose}$, there exists a unique solution $(u(x,t;\varepsilon),m(x,t;\varepsilon) )$ of $\eqref{main_equation}$. Let $(u(x,t;0),m(x,t;0) ) $ be the solution of $\eqref{main_equation}$ when $\varepsilon=0.$
	
	Let
	$$u^{(1)}:=\p_{\varepsilon_1}u|_{\varepsilon=0}=\lim\limits_{\varepsilon\to 0}\frac{u(x,t;\varepsilon)-u(x,t;0) }{\varepsilon_1},$$
	$$m^{(1)}:=\p_{\varepsilon_1}m|_{\varepsilon=0}=\lim\limits_{\varepsilon\to 0}\frac{m(x,t;\varepsilon)-m(x,t;0) }{\varepsilon_1}.$$
	
	The idea is that we consider a new system of $(u^{(1)},m^{(1)}).$ If  $F\in\mathcal{A}$, $g\in\mathcal{B}$, we have 
	\[
	(u(x,t;0),m(x,t;0) )=(0,0) 
	\] 
	and hence
	\begin{align*}
		&-\p_tu^{(1)}(x,t)-\Delta u^{(1)}(x,t)\\
		=& \lim\limits_{\varepsilon\to 0}\frac{1}{\varepsilon_1}[\frac{|\nabla u(x,t;\varepsilon)|^2-|\nabla u(x,t;0)|^2}{2}+ F(x,t,m(x,t;\varepsilon))-F(x,t;m(x,t;0)) ]\\
		=&\nabla u^{(1)}\cdot  (\lim\limits_{\varepsilon\to 0}\frac{\nabla u(x,t;\varepsilon)+\nabla u(x,t;0)}{2})+ \lim\limits_{\varepsilon\to 0}\frac{1}{\varepsilon_1}[ F^{(1)}(x,t)(m(x,t;\varepsilon)-m(x,t;0))]\\
		=&F^{(1)}(x,t)m^{(1)}(x,t).
	\end{align*}
	
	Similarly, we can compute
	\begin{align*}
		&\p_t m^{(1)}(x,t)-\Delta m^{(1)}(x,t)\\
		=&\lim\limits_{\varepsilon\to 0}\frac{1}{\varepsilon_1}[{\rm div} ( m(x,t;\varepsilon)\nabla u(x,t;\varepsilon)-m(x,t;0)\nabla u(x,t;0) )]\\
		=&\lim\limits_{\varepsilon\to 0}\frac{1}{\varepsilon_1} [ \nabla m(x,t;\varepsilon)\cdot\nabla u(x,t;\varepsilon)+m(x,t;\varepsilon)\Delta u(x,t;\varepsilon) -\\
		& \nabla m(x,t;0)\cdot\nabla u(x,t;0)-m(x,t;0)\Delta u(x,t;0)           ]\\
		=&0.
	\end{align*}
	Now, we have  that $(u_{j}^{(1)},m_{j}^{(1)} )$ satisfies the following system:
	\begin{equation}\label{linear l=1,eg}
		\begin{cases}
			-\p_tu^{(1)}(x,t)-\Delta u^{(1)}(x,t)= F^{(1)}(x,t)m^{(1)}(x,t),& \text{ in } \mathbb{T}^n\times (0,T),\medskip\\
			\p_t m^{(1)}(x,t)-\Delta m^{(1)}(x,t)=0,&\text{ in }\mathbb{T}^n\times (0,T),\medskip\\
			u^{(1)}_j(x,T)=G^{(1)}(x)m^{(1)}(x,T), & \text{ in } \mathbb{T}^n,\medskip\\
			m^{(1)}_j(x,0)=f_1(x). & \text{ in } \mathbb{T}^n.\\
		\end{cases}  
	\end{equation}
	Then we can define $$u^{(l)}:=\p_{\varepsilon_l}u|_{\varepsilon=0}=\lim\limits_{\varepsilon\to 0}\frac{u(x,t;\varepsilon)-u(x,t;0) }{\varepsilon_l},$$
	$$m^{(l)}:=\p_{\varepsilon_l}m|_{\varepsilon=0}=\lim\limits_{\varepsilon\to 0}\frac{m(x,t;\varepsilon)-m(x,t;0) }{\varepsilon_l},$$
	for all $l\in\mathbb{N}$ and obtain a sequence of similar systems.
	In the proof of Theorem $\ref{der g}$ in what follows, we recover the first Taylor coefficient of $F$ or $G$ by considering this new system $\eqref{linear l=1,eg}$. In order to recover the higher order Taylor coefficients, 
	we consider 
	\begin{equation}\label{eq:ht1}
		u^{(1,2)}:=\p_{\varepsilon_1}\p_{\varepsilon_2}u|_{\varepsilon=0},
		m^{(1,2)}:=\p_{\varepsilon_1}\p_{\varepsilon_2}m|_{\varepsilon=0}.
	\end{equation}
	
	By direct calculations, we have from \eqref{eq:ht1} that 
	\begin{equation}\label{eq:ht2}
		\begin{split}
			&-\p_tu^{(1,2)}(x,t)-\Delta u^{(1,2)}(x,t)\\
			=& -\nabla u^{(1)}\cdot \nabla u^{(2)}-\nabla u^{(1,2)}\cdot \nabla u(x,t;0) \\
			& +F_j^{(1)}m^{(1,2)}+F^{(2)}_j(x,t)m^{(1)}m^{(2)},
		\end{split}
	\end{equation}
	and
	\begin{equation}\label{eq:ht3}
		\begin{split}
			&\p_t m^{(1,2)}(x,t)-\Delta m^{(1,2)}(x,t)\\
			=& \p_{\varepsilon_1}\p_{\varepsilon_2}{\rm div} (m\nabla u)|_{\varepsilon=0}\medskip\\
			=&\nabla m^{(1,2)}\nabla u(x,t;0)+\nabla m(x,t;0)\nabla u^{(1,2)}+m^{(1,2)}\Delta u(x,t;0)+m(x,t;0)\Delta u^{(1,2)}\medskip\\
			&+ {\rm div} (m^{(1)}\nabla u^{(2)})+{\rm div}(m^{(2)}\nabla u^{(1)})\medskip\\
			=& {\rm div} (m^{(1)}\nabla u^{(2)})+{\rm div}(m^{(2)}\nabla u^{(1)}).
		\end{split}
	\end{equation}
	
	Combining \eqref{eq:ht2} and \eqref{eq:ht3}, we have the second order linearization as follows:
	\begin{equation}\label{linear l=1,2 eg}
		\begin{cases}
			-\p_tu^{(1,2)}-\Delta u^{(1,2)}(x,t)+\nabla u^{(1)}\cdot \nabla u^{(2)}\medskip\\
			\hspace*{3cm}= F^{(1)}(x,t)m^{(1,2)}+F^{(2)}(x,t)m^{(1)}m^{(2)},& \text{ in } \mathbb{T}^n\times(0,T),\medskip\\
			\p_t m^{(1,2)}-\Delta m^{(1,2)}= {\rm div} (m^{(1)}\nabla u^{(2)})+{\rm div}(m^{(2)}\nabla u^{(1)}) ,&\text{ in } \mathbb{T}^n\times (0,T),\medskip\\
			u^{(1,2)}(x,T)=G^{(1)}(x)m^{(1,2)}(x,T)+G^{(2)}(x)m^{(1)}m^{(2)}(x,T), & \text{ in } \mathbb{T}^n,\medskip\\
			m^{(1,2)}(x,0)=0, & \text{ in } \mathbb{T}^n.\\
		\end{cases}  	
	\end{equation}
	Notice that the non-linear terms of the system $\eqref{linear l=1,2 eg}$ depend on the first order linearised system $\eqref{linear l=1,eg}$. This shall be an important ingredient in the proof of Theorem $\ref{der g}$ in what follows. 
	
	Inductively, for $N\in\mathbb{N}$, we consider 
	\begin{equation*}
		u^{(1,2...,N)}=\p_{\varepsilon_1}\p_{\varepsilon_2}...\p_{\varepsilon_N}u|_{\varepsilon=0},
	\end{equation*}
	\begin{equation*}
		m^{(1,2...,N)}=\p_{\varepsilon_1}\p_{\varepsilon_2}...\p_{\varepsilon_N}m|_{\varepsilon=0}.
	\end{equation*}
	we can obtain a sequence of parabolic systems, which shall be employed again in determining the higher order Taylor coefficients of the unknowns $F$ and $G$.

	\subsection{Unique determination of single unknown function}
	
	We first present the proof of Theorem $\ref{der g}$.
	
	\begin{proof}[Proof of Theorem $\ref{der g}$]
		Consider the following systems for $j=1,2$:
		\begin{equation}\label{j=1,2for g}
			\begin{cases}
				-\p_tu_j(x,t)-\Delta u_j(x,t)+\frac 1 2 {|\nabla u_j|^2}= F(x,t,m_j(x,t)),& \text{ in }\mathbb{T}^n\times (0,T),\medskip\\
				\p_t m_j(x,t)-\Delta m_j(x,t)-{\rm div}(m_j(x,t)\nabla u_j(x,t))=0,&\text{ in }\mathbb{T}^n\times(0,T),\medskip\\
				u_j(x,T)=G_j(x,m_j(x,T)), & \text{ in } \mathbb{T}^n,\medskip\\
				m_j(x,0)=m_0(x), & \text{ in } \mathbb{T}^n.\\
			\end{cases}  		
		\end{equation}
		By the successive linearization procedure,  we first consider the case $N=1.$ Let
		$$u_{j}^{(1)}:=\p_{\varepsilon_1}u_{j}|_{\varepsilon=0},\quad m_{j}^{(1)}:=\p_{\varepsilon_1}m_{j}|_{\varepsilon=0}.$$
		Direct computations show that $(u_{j}^{(1)},v_{j}^{(1)} )$ satisfies the following system
		\begin{equation}\label{linear l=1for g}
			\begin{cases}
				-\p_tu_j^{(1)}(x,t)-\Delta u^{(1)}_j(x,t)= F^{(1)}(x,t)m_j^{(1)}(x,t),& \text{ in }\mathbb{T}^n\times(0,T),\medskip\\
				\p_t m^{(1)}_j(x,t)-\Delta m^{(1)}_j(x,t)=0,&\text{ in }\mathbb{T}^n\times(0,T),\medskip\\
				u^{(1)}_j(x,T)=G_j^{(1)}(x)m^{(1)}_j(x,T), & \text{ in } \mathbb{T}^n,\medskip\\
				m^{(1)}_j(x,0)=f_1(x), & \text{ in } \mathbb{T}^n.
			\end{cases}  	
		\end{equation}	
		We can solve the system \eqref{linear l=1for g} by first deriving $m^{(1)}_j$ and then obtaining $u^{(1)}_j.$ In doing so, we can obtain that the solution is
		$$ m_j^{(1)}(x,t)= \int_{\mathbb{R}^n}\Phi(x-y,t)f_{1}(y)\, dy,$$
		\begin{equation*}
			\begin{aligned}
				u_j^{(1)}(x,t)&=  \int_{\mathbb{R}^n}\Phi(x-y,T-t)G_j^{(1)}(y)m^{(1)}_j(y,T) )\, dy\\
				&+\int_{0}^{T-t}\int_{\mathbb{R}^n}\Phi(x-y,T-t-s)F^{(1)}(y,T-s)\overline{m}_j^{(1)}(y,s)\, dyds,
			\end{aligned}
		\end{equation*}
		where $\overline{m}_j^{(1)}(x,t)= m_j^{(1)}(x,T-t)$ and $\Phi$ is the  fundamental solution of the heat equation:
		\begin{equation}\label{eq:fund1}
			\Phi(x,t)= \frac{1}{(4\pi t)^{n/2}}e^{-\frac{|x|^2}{4t}}.
		\end{equation}
		
		Since  $\mathcal{M}_{G_1}=\mathcal{M}_{G_2}$, we have $$ u_1^{(1)}(x,0)=u_2^{(1)}(x,0),$$ for all $f_1\in C_+^{2+\alpha}(\mathbb{T}^n).$ This implies that
		$$ \int_{\mathbb{R}^n}\Phi(x-y,T)[G_1^{(1)}(y)m_1^{(1)}(y,T))-G_2^{(1)}(y)m_2^{(1)}(y,T)) ]\, dy=0.$$
		Noticing that $m_1^{(1)}(x,t)=m_2^{(1)}(x,t)$,  we  choose 
			$$m_1^{(1)}(x,T)=m_2^{(1)}(x,T)=\exp(-4\pi^2|\boldsymbol{\zeta}|^2T-2\pi \mathrm{i} \boldsymbol{\zeta}\cdot x)+M,$$ 
			where $\boldsymbol{\zeta}\in\mathbb{Z}^n, M\in\mathbb{N}.$ ( In this case, $f_1(x)\in C_+^{2+\alpha}(\mathbb{T}^n)$)
		
		By taking $M=1$ and $M=2$, respectively and then subtracting the resulting equations from one another, one can readily show that
		\begin{equation}
			\int_{\mathbb{R}^n}\Phi(x-y,T)[(G_1^{(1)}(y)-G_2^{(1)}(y))\exp(-2\pi \mathrm{i} \boldsymbol{\zeta}\cdot y)         ]\, dy=0,
		\end{equation}
		for all $\boldsymbol{\zeta}\in\mathbb{Z}^n.$
		Therefore $G_1^{(1)}(x)=G_2^{(1)}(x).$

		We proceed to consider the case $N=2.$ Let
		$$u_{j}^{(1,2)}:=\p_{\varepsilon_1}\p_{\varepsilon_2}u_{j}|_{\varepsilon=0},\quad
		m_{j}^{(1,2)}:=\p_{\varepsilon_1}\p_{\varepsilon_2}m_{j}|_{\varepsilon=0},$$
		and
		$$u_{j}^{(2)}:=\p_{\varepsilon_2}u_{j}|_{\varepsilon=0},\quad m_{j}^{(2)}:=\p_{\varepsilon_2}m_{j}|_{\varepsilon=0}.$$
		Then we can deal with the second-order linearization: 
		\begin{equation}\label{linear l=1,2}
			\begin{cases}
				-\p_tu_j^{(1,2)}(x,t)-\Delta u^{(1,2)}_j(x,t)+\nabla u_{j}^{(1)}\cdot \nabla u_{j}^{(2)}\\
				\hspace*{3cm} = F^{(1)}m_j^{(1,2)}+F^{(2)}(x,t)m_j^{(1)}m_j^{(2)},& \text{ in }\mathbb{T}^n\times (0,T),\\
				\p_t m^{(1,2)}_j(x,t)-\Delta m^{(1,2)}_j(x,t)\medskip \\
				\hspace*{3cm} = {\rm div} (m_{j}^{(1)}\nabla u_j^{(2)})+{\rm div}(m_j^{(2)}\nabla u_j^{(1)}) ,&\text{ in }\mathbb{T}^n\times(0,T),\medskip\\
				u^{(1,2)}_j(x,T)=G_j^{(1)}(x)m_j^{(1,2)}(x,T)+G_j^{(2)}(x)m_j^{(1)}m_j^{(2)}(x,T), & \text{ in } \mathbb{T}^n,\medskip\\
				m^{(1,2)}_j(x,0)=0, & \text{ in } \mathbb{T}^n.\\
			\end{cases}  	
		\end{equation}
		Since we have shown that $G_1^{(1)}(x)=G_2^{(1)}(x)$, we have 
		$$ u^{(1)}_1(x,t)= u^{(1)}_2(x,t), \, m^{(1)}_1(x,t)= m^{(1)}_2(x,t)$$
		by solving equation $\eqref{linear l=1for g}$.
		
		Then by the same argument in the case $N=1$ (considering $m_0=\varepsilon_2f_2$ ), we have
		
		$$u^{(2)}_1(x,t)=u^{(2)}_2(x,t), \, m^{(2)}_1(x,t)= m^{(2)}_2(x,t).$$
		
		Denote 
		\[
		p(x,t)={\rm div} (m_{j}^{(1)}\nabla u_j^{(2)})+{\rm div}(m_j^{(2)}\nabla u_j^{(1)}),\ \ q(x,t)= -\nabla u_{j}^{(1)}\cdot \nabla u_{j}^{(2)}.
		\] 
		Then we can also solve  system \eqref{linear l=1,2} as follows:
		\begin{equation*}
			m^{(1,2)}_j(x,t)=\int_{0}^{t} \int_{\mathbb{R}^n} \Phi(x-y,t-s)p(y,s)\, dyds,
		\end{equation*}
		\begin{equation*}
			\begin{aligned}
				u_j^{(1,2)}(x,t)= &\int_{\mathbb{R}^n}\Phi(x-y,T-t) [G_j^{(1)}(x)m_j^{(1,2)}(x,T)+G_j^{(2)}(x)m_j^{(1)}m_j^{(2)}(x,T) ]\, dy\\
				+&\int_{0}^{T-t}\int_{\mathbb{R}^n}\Phi(x-y,T-t-s)(F^{(2)}(y,T-s)m_j^{(1)}m_j^{(2)}(y,T-s) -\overline{q}(y,s))\, dyds,
			\end{aligned}
		\end{equation*}
		where $\overline{q}(y,s)=q(y,T-s).$
		Since  $$u_1^{(1,2)}(x,0)= u_2^{(1,2)}(x,0),$$ we have	$$ \int_{\mathbb{R}^n}\Phi(x-y,T)[G_1^{(2)}(y)m_1^{(1)}(y,T)-G_2^{(1)}(y)m_2^{(1)}(y,T) ]\, dy=0.$$

		Next, by a similar argument in the case $N=1$, we can prove that $G^{(2)}_1(x)=G^{(2)}_2(x). $ 
		Finally, by the mathematical induction, we can show the same result for $N\geq 3$. That is, for any $k\in\mathbb{N},$ we have $G^{(k)}_1(x)=G^{(k)}_2(x).$ 
		The proof is complete. 
	\end{proof}
	
	\subsection{Simultaneous recovery results for inverse problems}
	In this section, we aim to determinate $F$ and $G$ simultaneously. To that end, we first derive an auxiliary lemma as follows. 
	
		\begin{lem}\label{dense}
			Let $u$ be a solution of the heat equation 
			\begin{equation}\label{per heat}
				\begin{cases}
					\p_t u(x,t)-\Delta u(x,t)=0    &\text{ in } \mathbb{T}^n\times(0,T),\\
					u(x,0)=u_0(x)                  &\text{ in } \mathbb{T}^n.
				\end{cases}
			\end{equation}
			Let $f(x)\in C^{2+\alpha}(\mathbb{T}^n)$ for some $\alpha\in(0,1)$. Suppose 
			\begin{equation}\label{fuv=0}
				\int_{\mathbb{T}^n\times(0,T)} f(x)u(x,t)\, dxdt=0,	
			\end{equation}
			for all $u_0\in C_+^{\infty}(\mathbb{T}^n)$. Then one has $f=0.$
	\end{lem}
	\begin{proof}
		Let  $\boldsymbol {\xi}\in\mathbb{Z}^n$ and $M\in\mathbb{N}$. It is directly verified that
		$$
		u(x,t)=\exp(- 2\pi\mathrm{i}\boldsymbol {\xi}\cdot x-4\pi^2|\boldsymbol {\xi}|^2t )+M, \quad  \mathrm{i}:=\sqrt{-1},
		$$
		is a solution of $\eqref{per heat}$ with initial value 
		$$u_0(x)= \exp(- 2\pi\mathrm{i}\boldsymbol {\xi}\cdot x)+M\geq0.$$
		Then $\eqref{fuv=0}$ implies that
		$$\int_{\mathbb{T}^n}  \frac{1-\exp(-4\pi^2|\boldsymbol {\xi}|^2T)}{4\pi^2|\boldsymbol {\xi}|^2}f(x)e^{-2\pi \mathrm{i}\boldsymbol {\xi}\cdot x } dx+MT\int_{\mathbb{T}^n}f(x)dx=0.$$
		By taking $M=1$ and $M=2$,respectively, we have 
		$$ \int_{\mathbb{T}^n}f(x)e^{-2\pi \mathrm{i}\boldsymbol {\xi}\cdot x } dx=0.$$
		Hence, the Fourier series of $f(x)$ is $0$. Since $f(x)\in C^{2+\alpha}(\mathbb{T}^n)$, its Fourier series converges to $f(x)$ uniformly.
		Therefore, $f(x)=0.$
	\end{proof}
	
	We are now in a position to present the proof of Theorem $\ref{der F,g}$.
	\begin{proof}[Proof of Thoerem $\ref{der F,g}$]
		Consider the following systems 
		\begin{equation}\label{j=1,2for Fg}
			\begin{cases}
				-\p_tu_j(x,t)-\Delta u_j(x,t)+\frac 1 2 {|\nabla u_j|^2}= F_j(x,m_j(x,t)),& \text{ in }\mathbb{T}^n\times(0,T),\medskip\\
				\p_t m_j(x,t)-\Delta m_j(x,t)-{\rm div}(m_j(x,t)\nabla u_j(x,t))=0,&\text{ in }\mathbb{T}^n\times(0,T),\medskip\\
				u_j(x,T)=G_j(x,m_j(x,T)), & \text{ in } \mathbb{T}^n,\medskip\\
				m_j(x,0)=m_0(x), & \text{ in } \mathbb{T}^n.
			\end{cases}  		
		\end{equation}
		Following a similar method we used in the proof of Theorem $\ref{der g}$, we let 
		$$m_0(x;\varepsilon)=\sum_{l=1}^{N}\varepsilon_lf_l,$$
		where $f_l\in C_+^{2+\alpha}(\mathbb{T}^n)$ and $\varepsilon=(\varepsilon_1,\varepsilon_2,...,\varepsilon_N)\in\mathbb{R}_+^N$ with 
		$|\varepsilon|=\sum_{l=1}^{N}|\varepsilon_l|$ small enough. 
		
		Consider the case $N=1.$ Let
		$$u_{j}^{(1)}:=\p_{\varepsilon_1}u_{j}|_{\varepsilon=0},$$
		$$m_{j}^{(1)}:=\p_{\varepsilon_1}m_{j}|_{\varepsilon=0}.$$
		Then direct computations imply that $(u_{j}^{(1)},v_{j}^{(1)} )$ satisfies the following system:
		\begin{equation}\label{linear l=1for F g}
			\begin{cases}
				-\p_tu_j^{(1)}(x,t)-\Delta u^{(1)}_j(x,t)= F_j^{(1)}(x)m_j^{(1)}(x,t),& \text{ in }\mathbb{T}^n\times(0,T),\medskip \\
				\p_t m^{(1)}_j(x,t)-\Delta m^{(1)}_j(x,t)=0,&\text{ in }\mathbb{T}^n\times (0,T),\medskip \\
				u^{(1)}_j(x,T)=G_j^{(1)}(x)m^{(1)}_j(x,T), & \text{ in } \mathbb{T}^n,\medskip\\
				m^{(1)}_j(x,0)=f_1(x), & \text{ in } \mathbb{T}^n.
			\end{cases}  	
		\end{equation}
		Then we have $ m_1^{(1)}=m_2^{(1)}:=m^{(1)}(x,t)$	.
		Let $ \overline{u}=u_1^{(1)}-u_2^{(1)}$, $\eqref{linear l=1for F g}$ implies that 
		\begin{equation}\label{u1-u2 }
			\begin{cases}
				&-\p_t\overline{u}-\Delta\overline{u}= (F_1^{(1)}-F_2^{(1)})m^{(1)}(x,t),\medskip\\
				&\overline{u}(x,T)=(G_1^{(1)}-G_2^{(1)})m^{(1)}(x,T).
			\end{cases}
		\end{equation}
		Now let $w$ be a solution to the heat equation $\p_t w(x,t)-\Delta w(x,t)=0$ in $\mathbb{T}^n$. Then
		\begin{equation}
			\begin{aligned}
				&\int_Q (F_1^{(1)}-F_2^{(1)})m^{(1)}(x,t)w\, dxdt\medskip\\
				=&\int_Q (-\p_t\overline{u}-\Delta\overline{u})w\, dxdt\medskip\\
				=&\int_{\mathbb{T}^n} (\overline{u}w)\big|_0^T\, dx +\int_Q \overline{u}\p_tw- \overline{u}\Delta w\medskip\\
				=&  \int_{\mathbb{T}^n} (\overline{u}w)\big|_0^T\, dx.
			\end{aligned}	
		\end{equation}
		
		Since $\mathcal{M}_{F_1,G_1}=\mathcal{M}_{F_2,G_2}$, we have $$\overline{u}(x,0)=0.$$ It follows that
		\begin{equation}\label{integral by part}
			\int_Q (F_1^{(1)}-F_2^{(1)})m^{(1)}(x,t)w(x,t)\, dxdt= \int_{\mathbb{T}^n} w(x,T)(G_1^{(1)}-G_2^{(1)})m^{(1)}(x,T)\, dx,
		\end{equation}
		for all solutions $w(x,t),m^{(1)}(x,t)$ of the heat equation in $\mathbb{T}^n$. 
		
		Here, we cannot apply Lemma $\ref{dense}$ directly. Nevertheless, we use the same construction.
		Let $\boldsymbol {\xi_1},\boldsymbol {\xi_2}\in\mathbb{Z}^n\backslash\{0\}$, $M\in\mathbb{N}^*$ and $\boldsymbol {\xi}=\boldsymbol {\xi_1}+\boldsymbol {\xi_2}$ .
		Let
		$$w(x,t)=\exp(- 2\pi\mathrm{i}\boldsymbol {\xi_1}\cdot x-4\pi^2|\boldsymbol {\xi_1}|^2t ),$$
		and
		
			$$m(x,t)=\exp(- 2\pi\mathrm{i}\boldsymbol {\xi_2}\cdot x-4\pi^2|\boldsymbol {\xi_2}|^2t )+M.$$
		Then the left hand side of $\eqref{integral by part}$ is
		\begin{equation}\label{Fourier1}
			\begin{aligned}
				&\int_{\mathbb{T}^n} \frac{1-\exp(-4\pi^2T(|\boldsymbol {\xi_1}|^2+|\boldsymbol {\xi_2}|^2))}{4\pi^2( |\boldsymbol {\xi_1}|^2+|\boldsymbol {\xi_2}|^2)}(F_1^{(1)}-F_2^{(1)} )e^{-2\pi \mathrm{i}\boldsymbol {\xi}\cdot x }\, dx\\
				+&M \frac{1-\exp(-4\pi^2T|\boldsymbol {\xi_1}|^2)}{4\pi^2|\boldsymbol {\xi_1}|^2}   \int_{\mathbb{T}^n} (F_1^{(1)}-F_2^{(1)} )e^{-2\pi \mathrm{i}\boldsymbol {\xi_1}\cdot x }\,dx.
			\end{aligned}
		\end{equation}
		And the right hand side is 
		\begin{equation}\label{Fourier2}
			\begin{aligned}
				&\int_{\mathbb{T}^n}\exp(-4\pi^2T(|\boldsymbol {\xi_1}|^2+|\boldsymbol {\xi_2}|^2)) (G_1^{(1)}-G_2^{(1)})e^{-2\pi \mathrm{i}\boldsymbol {\xi}\cdot x }\, dx\\
				+&M\exp(-4\pi^2T|\boldsymbol {\xi_1}|^2)\int_{\mathbb{T}^n} (G_1^{(1)}-G_2^{(1)})e^{-2\pi \mathrm{i}\boldsymbol {\xi_1}\cdot x }\, dx.
			\end{aligned}	
		\end{equation}

			By taking $M=1$ and $M=2$, respectively and then subtracting the resulting equations from one another, one can readily show that
			$$ \frac{1-\exp(-4\pi^2T|\boldsymbol {\xi_1}|^2)}{4\pi^2|\boldsymbol {\xi_1}|^2}   \int_{\mathbb{T}^n} (F_1^{(1)}-F_2^{(1)} )e^{-2\pi \mathrm{i}\boldsymbol {\xi_1}\cdot x }\,dx= \exp(-4\pi^2T|\boldsymbol {\xi_1}|^2)\int_{\mathbb{T}^n} (G_1^{(1)}-G_2^{(1)})e^{-2\pi \mathrm{i}\boldsymbol {\xi_1}\cdot x }\, dx. $$
		
		Then $\eqref{integral by part}$,$\eqref{Fourier1}$ and $\eqref{Fourier2}$ readily yields that
		$$\frac{1-\exp(-4\pi^2T(|\boldsymbol {\xi_1}|^2+|\boldsymbol {\xi_2}|^2))}{4\pi^2( |\boldsymbol {\xi_1}|^2+|\boldsymbol {\xi_2}|^2)}a_{\boldsymbol {\xi}  }+\exp(-4\pi^2T(|\boldsymbol {\xi_1}|^2+|\boldsymbol {\xi_2}|^2))b_{\boldsymbol {\xi}}=0.$$
		
		For a given $\boldsymbol {\xi}\in\mathbb{Z}^n $, there exist $\boldsymbol {\xi_1},\boldsymbol {\xi_2},\boldsymbol {\xi_1}',\boldsymbol {\xi_2}' \in\mathbb{Z}^n\backslash\{0\}$ such that $\boldsymbol {\xi}=\boldsymbol {\xi_1}+\boldsymbol {\xi_2}=\boldsymbol {\xi_1}'+\boldsymbol {\xi_2}'$ and $|\boldsymbol {\xi_1}|^2+|\boldsymbol {\xi_2}|^2\neq |\boldsymbol {\xi_1}'|^2+|\boldsymbol {\xi_2}'|^2 .$ Therefore, $a_{\boldsymbol {\xi}}=b_{\boldsymbol {\xi}}=0$ for all $\boldsymbol {\xi}\in\mathbb{Z}^n$. 
		Notice that 
		
		It  follows that $F_1^{(1)}-F_2^{(1)}=G_1^{(1)}-G_2^{(1)}=0$.

		Next, we consider the case $N=2.$ Let
		\begin{equation}\label{eq:ss1}
			u_{j}^{(1,2)}:=\p_{\varepsilon_1}\p_{\varepsilon_2}u_{j}|_{\varepsilon=0},\quad
			m_{j}^{(1,2)}:=\p_{\varepsilon_1}\p_{\varepsilon_2}m_{j}|_{\varepsilon=0},
		\end{equation}
		and
		\begin{equation}\label{eq:ss2}
			u_{j}^{(2)}:=\p_{\varepsilon_2}u_{j}|_{\varepsilon=0},\quad m_{j}^{(2)}:=\p_{\varepsilon_2}m_{j}|_{\varepsilon=0}.
		\end{equation}
		By the second-order linearization in \eqref{eq:ss1} and \eqref{eq:ss2},  we can obtain
		\begin{equation}
			\begin{cases}
				-\p_tu_j^{(1,2)}(x,t)-\Delta u^{(1,2)}_j(x,t)+\nabla u_{j}^{(1)}\cdot \nabla u_{j}^{(2)}\\
				\hspace*{3cm} = F_j^{(1)}m_j^{(1,2)}+F^{(2)}_j(x)m_j^{(1)}m_j^{(2)},& \text{ in }\mathbb{T}^n\times(0,T),\medskip\\
				\p_t m^{(1,2)}_j(x,t)-\Delta m^{(1,2)}_j(x,t)\\
				\hspace*{3cm}=  {\rm div} (m_{j}^{(1)}\nabla u_j^{(2)})+{\rm div}(m_j^{(2)}\nabla u_j^{(1)}) ,&\text{ in }\mathbb{T}^n\times (0,T),\medskip\\
				u^{(1,2)}_j(x,T)=G^{(1)}(x)m_j^{(1,2)}(x,T)+G^{(2)}(x)m_j^{(1)}m_j^{(2)}(x,T), & \text{ in } \mathbb{T}^n,\medskip\\
				m^{(1,2)}_j(x,0)=0, & \text{ in } \mathbb{T}^n.
			\end{cases}  	
		\end{equation}
		By following a similar argument in the case $N=1$ , we have 
		$$ u^{(1)}_1(x,t)= u^{(1)}_2(x,t),  u^{(2)}_1(x,t)=u^{(2)}_2(x,t),$$
		and
		$$ m^{(1)}_1(x,t)= m^{(1)}_2(x,t) ,  m^{(2)}_1(x,t)= m^{(2)}_2(x,t).$$
		
		Let $\overline{u}^2(x,t)=u_1^{(1,2)}(x,t)-u_2^{(1,2)}(x,t) $. We have
		\begin{equation}\label{u1-u2,2 }
			\begin{cases}
				&-\p_t\overline{u}^2-\Delta\overline{u}^2= (F_1^{(1)}-F_2^{(1)})m^{(1)}(x,t)m_1^{(2)}(x,t),\medskip\\
				&\overline{u}(x,T)=(G_1^{(1)}-G_2^{(1)})m^{(1)}(x,T)m_1^{(2)}(x,t).
			\end{cases}
		\end{equation}
		Let $w$ be a solution of the heat equation $\p_t w(x,t)-\Delta w(x,t)=0$ in $\mathbb{T}^n$. Then by following a similar argument in the case $N=1$, we can show that 
		\begin{equation}\label{integral by part2}
			\begin{split}
				&	\int_Q (F_1^{(2)}-F_2^{(2)})m^{(1)}m_1^{(2)}w(x,t)\, dxdt\\
				=& \int_{\mathbb{T}^n} w(x,T)(G_1^{(2)}-G_2^{(2)})m^{(1)}(x,T)m_1^{(2)}(x,T)\, dx.
			\end{split} 
		\end{equation}
		To proceed further, by using the construction in Lemma $\ref{dense}$ again, we have from \eqref{integral by part2} that
		\[
		F_1^{(2)}-F_2^{(2)}=G_1^{(2)}-G_2^{(2)}=0.
		\]
		
		Finally, via a mathematical induction, we can derive the same result for $N\geq 3$. That is, for any $k\in\mathbb{N},$ we have $$F^{(k)}_1(x)-F^{(k)}_2(x)=G^{(k)}_1(x)-G^{(k)}_2(x)=0.$$ Hence, 	
		$$(F_1(x,z),F_2(x,z))=(G_1(x,z),G_2(x,z)),\text{  in  } \mathbb{R}^n\times \mathbb{R}.$$ 
		
		The proof is complete. 
	\end{proof}
	
	\begin{rmk}
		Theorem  $\ref{der F,g}$ is not strictly stronger than Theorem $\ref{der g}$. We need $F(x,z)$ is independ of $t$ in the proof of  Theorem  $\ref{der F,g}$ but we do not need this condition in the proof of  Theorem $\ref{der g}$.
	\end{rmk}
		\begin{rmk}
			In the proof of Theorem $\ref{der F,g}$, we arrived at a decoupled system after applying the linearization technique. However, we cannot simply apply existing results in inverse problems for a single parabolic equation. In fact, for a single parabolic equation, it is impossible to determine the source term $f$ by the corresponding boundary measurement. For a simple illustration, we let $h(x)\in C_0^\infty(Q)$, and consider the following two parabolic equations for a given $f\in C(Q)$, 
			\[
			\partial_t u-\Delta u=f \quad \mbox{and}\quad \partial_t\widetilde{u}-\Delta\widetilde{u}=\widetilde{f},\quad \widetilde{f}:=f+(\partial_t h-\Delta h). 
			\] 
			It can be directly verified that $u$ and $\widetilde{u}$ possess the same boundary data, though $f\equiv\widetilde f$ in general. Hence, the proof of Theorem $\ref{der g}$ makes advantageous use on the peculiar structures of the MFG system. The same fact holds for the proofs of Theorems $\ref{der F 2}$ and $\ref{der g2}$ in what follows. 
	\end{rmk}
	\section{Inverse Problems for MFGs with General Lagrangians}
	In the previous sections, we established the unique identifiability results for the inverse problems by assuming that the Hamiltonian involved is of a quadratic form, which represents a kinetic energy. In this section, we show that one can extend a large part of the previous results to the case with general Lagrangians if $F$ is independent of $t$.

	In what follows, we let $T>0$ and $n\in\mathbb{N}$ and consider the following system of nonlinear PDEs : 
	\begin{equation}\label{general H}
		\begin{cases}
			-\p_tu(x,t)-\Delta u(x,t)+ H(x,\nabla u)= F(x,m(x,t)),& \text{ in }\mathbb{T}^n\times (0,T),\medskip\\
			\p_t m(x,t)-\Delta m(x,t)-{\rm div}(m(x,t) H_p (x,\nabla u))=0,&\text{ in }\mathbb{T}^n\times (0,T),\medskip \\
			u(x,T)=G(x,m_T), & \text{ in } \mathbb{T}^n,\medskip \\
			m(x,0)=m_0(x), & \text{ in } \mathbb{T}^n.
		\end{cases}  	
	\end{equation}
	We study the inverse problem \eqref{eq:ip1}-\eqref{eq:ip2} associated with \eqref{general H}. In order to apply the method developed in the previous sections to this general case, we first introduce a new analytic class. 
	
	\begin{defi}
		Let $H(x,z_1,z_2,...,z_n)$ be a function mapping from $\mathbb{R}\times\mathbb{C}^n $ to $\mathbb{C}$. We say that $H$ is admissible and write $H \in \mathcal{I}$ if it fulfils the following conditions:
		\begin{enumerate}
			
			\item[(1)]~The map $(z_1,z_2,...,z_n)\to H(\cdot,z_1,z_2,...,z_n)$ is holomorphic with value in $C^{2+\alpha}(\mathbb{T}^n)$, $\alpha\in(0,1)$;
			
			\item[(2)] $H(x,0)=0, $ for all $x\in\mathbb{T}^n.$ 
		\end{enumerate}
		It is clear that $H$ can be expanded into a power series:
		\begin{equation}\label{eq:sss1}
			H(x,z)=\sum_{|\beta|=1}^{\infty} H^{(\beta)}(x)\frac{z^{\beta}}{k!},
		\end{equation}
		where $ H^{(\beta)}(x)\in C^{2+\alpha}(\mathbb{T}^n)$ and $\beta$ is a muti-index.
	\end{defi}
	
	Similar to our discussion in Remark~\ref{rem:1}, we always assume that the coefficient functions $H^{(\beta)}$ in \eqref{eq:sss1} are real-valued.  We first state the main theorems of the results for the inverse problems associated with \eqref{general H}. The corresponding proofs are given in Section $\ref{proof H}$.
	
	\begin{thm}\label{der F 2}
		Assume $F_j\in\mathcal{B}$ ($j=1,2$), $G\in\mathcal{B}$ and $H\in\mathcal{I}$. Let $\mathcal{M}_{F_j}$ be the measurement map associated to
		the following system ($j=1,2$):
		\begin{equation}
			\begin{cases}
				-\p_tu(x,t)-\Delta u(x,t)+ H(x,\nabla u)= F_j(x,m(x,t)),& \text{ in }\mathbb{T}^n\times (0,T),\medskip\\
				\p_t m(x,t)-\Delta m(x,t)-{\rm div}(m(x,t) H_p (x,\nabla u))=0,&\text{ in }\mathbb{T}^n\times (0,T),\medskip\\
				u(x,T)=G(x,m_T), & \text{ in } \mathbb{T}^n,\medskip\\
				m(x,0)=m_0(x), & \text{ in } \mathbb{T}^n.
			\end{cases}  	
		\end{equation}
		If for any $m_0\in C^{2+\alpha}(\mathbb{T}^n)\cap\mathcal{O}_a$,  one has 
		$$\mathcal{M}_{F_1}(m_0)=\mathcal{M}_{F_2}(m_0),$$
		then it holds that
		$$F_1(x,z)=F_2(x,z) \text{  in  } \mathbb{T}^n\times \mathbb{R}.$$ 
	\end{thm}
	\begin{thm}\label{der g2}
		Assume $F \in\mathcal{B}$, $G_j\in\mathcal{B}$ ($j=1,2$) and $H\in\mathcal{I}$. Let $\mathcal{M}_{G_j}$ be the measurement map associated to the following system ($j=1,2$): 
		\begin{equation}
			\begin{cases}
				-\p_tu(x,t)-\Delta u(x,t)+ H(x,\nabla u)= F(x,t,m(x,t)),& \text{ in }\mathbb{T}^n\times (0,T),\medskip\\
				\p_t m(x,t)-\Delta m(x,t)-{\rm div}(m(x,t) H_p (x,\nabla u))=0,&\text{ in } \mathbb{T}^n\times (0,T),\medskip\\
				u(x,T)=G_j(x,m_T), & \text{ in } \mathbb{T}^n,\medskip\\
				m(x,0)=m_0(x), & \text{ in } \mathbb{T}^n.
			\end{cases}  	
		\end{equation}
		If for any $m_0\in C^{2+\alpha}(\mathbb{T}^n)\cap\mathcal{O}_a$,  one has 
		$$\mathcal{M}_{G_1}(m_0)=\mathcal{M}_{G_2}(m_0),$$
		then it holds that
		$$G_1(x,z)=G_2(x,z) \text{  in  } \mathbb{T}^n\times \mathbb{R}.$$ 
	\end{thm}
	\subsection{Well-posedness of the general system}
	
	\begin{lem}\label{localwellpose2}
		Suppose  $F,G\in\mathcal{B}$ ,$H\in\mathcal{I}$. Then
		there exist $\delta>0$, $C>0$ such that for any $m_0\in B_{\delta}(\mathbb{T}^n) :=\{m_0\in C^{\alpha}(\mathbb{T}^n): \|m_0\|_{C^{2+\alpha}(\mathbb{T}^n)}\leq\delta \}$, the MFG system $\eqref{general H}$ has a solution $u = u_{m_0} \in
		C^{2+\alpha,1+\frac{\alpha}{2}}(Q)$ which satisfies
		\begin{equation}\label{eq:nn4}
			\|u\|_{C^{2+\alpha,1+\frac{\alpha}{2}}(Q}+ \|m\|_{C^{2+\alpha,1+\frac{\alpha}{2}}(Q)}\leq C\|m_0\|_{ C^{2+\alpha}(\mathbb{T}^n)}. 
		\end{equation}
		Furthermore, the solution $(u,m)$ is unique within the class
		\begin{equation}\label{eq:nn5}
			\{ (u,m)\in  C^{2+\alpha,1+\frac{\alpha}{2}}(Q)^2 : \|(u,m)\|_{ C^{2+\alpha,1+\frac{\alpha}{2}}(Q)^2}\leq C\delta \},
		\end{equation}
		where
		\begin{equation}\label{eq:nn6} 
			\|(u,m)\|_{ C^{2+\alpha,1+\frac{\alpha}{2}}(Q)^2}:= \|u\|_{C^{2+\alpha,1+\frac{\alpha}{2}}(Q)}+ \|m\|_{C^{2+\alpha,1+\frac{\alpha}{2}}(Q)},
		\end{equation}
		and it depends holomorphically on $m_0\in C^{2+\alpha}(\mathbb{T}^n)$.

	\end{lem}
	
	The proof of Lemma $\ref{localwellpose2}$ follows from a similar argument to that of Lemma $\ref{local_wellpose}$. We choose to skip it.

	\subsection{Proofs of Theorem $\ref{der F 2}$ and $\ref{der g2}$}\label{proof H}
	
	We first introduce the general heat kernel to recover the unknown functions in a parabolic system. The  construction and basic properties of the general heat kernel can be found in \cite{ito11}.
	
	\begin{lem}\label{general heat ker}
		Let $F_1,F_2,f\in C^{2+\alpha,1+\frac{\alpha}{2}}(Q)$, $g\in C^{2+\alpha}(\mathbb{T}^n)$ and $A(x)\in C^{2+\alpha,1+\frac{\alpha}{2}}(\mathbb{T}^n)^n$. Consider the following system
		\begin{equation}
			\begin{cases}
				\p_tu_i(x,t)-\Delta u_i(x,t)+ A(x)\cdot \nabla u_i= F_i(x)v(x,t)+f(x,t),& \text{ in }\mathbb{T}^n\times (0,T),\medskip\\
				u_i(x,0)=  g(x) , & \text{ in } \mathbb{T}^n .\\
			\end{cases} 
		\end{equation}
		Suppose for any $v(x,t)\in  C^{2+\alpha,1+\frac{\alpha}{2}}(Q)$, we have $u_1(x,T;v)=u_2(x,T;v)$. Then it holds that $F_1=F_2.$
		\begin{proof}
			Let $L=\partial_t-\Delta+A\cdot\nabla(\cdot)$ and $K(x,y,t)$ be the solution of the following Cauchy problem
			\begin{equation*}
				\begin{cases}
					&L (K(x,t))=0,\ \ t>0,\ \ x\in\mathbb{R}^n,\medskip\\
					&K(x,0)=\delta(0).
				\end{cases}
			\end{equation*}
			Then one has that 
			\begin{equation*}
				\begin{aligned}
					u_i(x,t)=&\int_{\mathbb{T}^n}K(x-y,t)g(y)dy\\
					+&\int_{0}^t\int_{\mathbb{T}^n}K(x-y,t-s)(F_i(y)v(y,s)+f(y,s))\, dyds.
				\end{aligned}
			\end{equation*}
			Since we have $u_1(x,T;v)=u_2(x,T;v)$, it follows that
			\begin{equation}\label{implies F1=F2}
				\int_{0}^T\int_{\mathbb{T}^n}K(x-y,T-s)(F_1(y)-F_2(y))v(y,s)\, dyds=0.
			\end{equation}
			By absurdity, we assume that there is $y_0\in \mathbb{T}^n$ such that $ F_1(y_0)\neq F_2(y_0)$. Then there is a neighborhood $U$ of $ y_0 $ such that $F_1-F_2>0$ or $F_1-F_2<0$ in $U$. Since $K(x-y,T-s)>0$  and $\eqref{implies F1=F2}$ holds for all $v\in  C^{2+\alpha,1+\frac{\alpha}{2}}(Q)$. We may choose $v$ such that 
			$K(x-y,T-s)(F_1(y)-F_2(y))v(y,s)>0$ in $U$ and $K(x-y,T-s)(F_1(y)-F_2(y))v(y,s)=0$ in $\mathbb{T}^n\backslash U$. It is a contradiction. Therefore, we have $F_1=F_2.$
			
			The proof is complete. 
		\end{proof}
	\end{lem}
	
	Before we present the proofs for Theorems $\ref{der F 2}$ and $\ref{der g2}$, we first perform the higher order
	linearization for the MFG system $\eqref{general H}$, which follows a similar strategy to that developed in Section $\ref{HLM}$.
	Let 
	$$m_0(x;\varepsilon)=\sum_{l=1}^{N}\varepsilon_lf_l,$$
	where $f_l\in C_+^{2+\alpha}(\mathbb{T}^n)$ and $\varepsilon=(\varepsilon_1,\varepsilon_2,...,\varepsilon_N)\in\mathbb{R}_+^N$ with 
	$|\varepsilon|=\sum_{l=1}^{N}|\varepsilon_l|$ small enough. Then by Lemma $\ref{localwellpose2}$, there exists a unique solution $(u(x,t;\varepsilon),m(x,t;\varepsilon) )$ of $\eqref{general H}$. Let $(u(x,t;0),m(x,t;0) ) $ be the solution of $\eqref{general H}$ when $\varepsilon=0.$
	Notice that if  $H\in\mathcal{I}, $ then $(u(x,t;0),m(x,t;0) ) =(0,0).$

	Let
	$$u^{(1)}:=\p_{\varepsilon_1}u|_{\varepsilon=0},$$
	$$m^{(1)}:=\p_{\varepsilon_1}m|_{\varepsilon=0}.$$
	Suppose $H\in\mathcal{I}$, $F\in\mathcal{A}$ and $G\in\mathcal{B}, $ we have
	\begin{equation}\label{compute for H1}
		\begin{aligned}
			&\p_t m^{(1)}_j(x,t)-\Delta m^{(1)}_j(x,t)\\
			=&\lim\limits_{\varepsilon\to 0}\frac{1}{\varepsilon_l} [ -H(x,\nabla u(x,t;\varepsilon)) +H(x;u(x,t;0))+ F(x,u(x,t;\varepsilon))-F(x;u(x,t:0))    ]\\
			=&\lim\limits_{\varepsilon\to 0}\frac{1}{\varepsilon_l} [ \sum_{|\beta|=1}^{\infty} H^{(\beta)}(x)\frac{z^{\beta}}{k!}]+F^{(1)}(x)m_j^{(1)}(x,t)\\
			=& -A^{(1)}(x)\cdot \nabla u+F^{(1)}(x)m_j^{(1)}(x,t),
		\end{aligned}
	\end{equation}
	where $A^{(1)}(x)=(H^{(1,0,0,...,0)}(x),H^{(0,1,0,...,0)}(x),...,H^{(0,0,...,1)}(x) ).$

	Moreover, we have
	\begin{equation}\label{compute for H2}
		\begin{aligned}
			&\p_{\varepsilon_1}   {\rm div}(m(x,t) H_p (x,\nabla u))     |_{\varepsilon=0}\\
			=&\p_{\varepsilon_1} {\rm div} (  m(x,t)  A^{(1)}(x) )+ \p_{\varepsilon_1}  {\rm div}(m(x,t)   B^{(1)}(x) \cdot \nabla u)|_{\varepsilon=0}\\
			=&\p_{\varepsilon_1} {\rm div} (  m(x,t)  A^{(1)}(x) ),
		\end{aligned}
	\end{equation}
	where 
	\[ 
	\begin{split}
		B^{(1)}(x)=&(\sum_{|\beta|=1}H^{(1,\beta)}(x),\sum_{|\alpha|+|\beta|=1,\alpha\in\mathbb{R}}H^{(\alpha,1,\beta)}(x),\\
		&\sum_{|\alpha|+|\beta|=1,\alpha\in\mathbb{R}^2}H^{(\alpha,1,\beta)}(x),.....,\sum_{|\alpha|=1,\alpha\in\mathbb{R}^{n-1}}H^{(\alpha,1)}(x)   ).
	\end{split}
	\]
	
	Hence, we can see that $(u^{(1)},m^{(1)} )$ satisfies the following system:
	\begin{equation}\label{H linear l=1 eg}
		\begin{cases}
			-\p_tu^{(1)}(x,t)-\Delta u^{(1)}(x,t)+ A^{(1)}(x)\cdot \nabla u= F^{(1)}(x)m^{(1)}(x,t),& \text{ in }\mathbb{T}^n\times (0,T),\medskip\\
			\p_t m^{(1)}(x,t)-\Delta m^{(1)}(x,t)-{\rm div} ( m^{(1)}(x,t) A^{(1)}(x))=0,&\text{ in }\mathbb{T}^n\times (0,T),\medskip\\
			u^{(1)}(x,T)=G^{(1)}(x)m^{(1)}(x,T), & \text{ in } \mathbb{T}^n,\medskip\\
			m^{(1)}(x,0)=f_1(x), & \text{ in } \mathbb{T}^n,
		\end{cases}  	
	\end{equation}
	
	Here, we make a key observation that the non-linear terms and source terms in higher-order
	linearization system only depend on the solutions of the lower-order
	linearization system. Hence, as an illustrative case for our argument, we only compute the second order linearization system.
	Let
	$$u^{(1,2)}:=\p_{\varepsilon_1}\p_{\varepsilon_2}u|_{\varepsilon=0},
	m^{(1,2)}:=\p_{\varepsilon_1}\p_{\varepsilon_2}m|_{\varepsilon=0},$$
	and
	$$u^{(2)}:=\p_{\varepsilon_2}u|_{\varepsilon=0},m^{(2)}:=\p_{\varepsilon_2}m|_{\varepsilon=0}.$$
	Recall the derivation of the system $\eqref{linear l=1,2 eg}$ in Section $\ref{HLM}$. By direct calculations, we have
	\begin{equation}\label{compute H 12 eg}
		\begin{aligned}
			&-\p_tu^{(1,2)}-\Delta u^{(1,2)}\\
			=&-\p_{\varepsilon_1}\p_{\varepsilon_2}H(x,\nabla u)|_{\varepsilon=0}+F^{(1)}(x)m^{(1,2)}+F^{(2)}(x)m^{(1)}m^{(2)}\\
			=&-\p_{\varepsilon_1}\p_{\varepsilon_2}(\sum_{|\beta|=1}^{2} H^{(\beta)}(x)\frac{z^{\beta}}{k!})|_{\varepsilon=0}+F^{(1)}(x)m^{(1,2)}+F^{(2)}(x)m^{(1)}m^{(2)}\\
			=&-A^{(1)}\cdot\nabla u_j^{(1,2)}-\sum_{|\beta|=2}H^{(\beta)}(x)u_j^{(1)}u_j^{(2)}++F^{(1)}(x)m^{(1,2)}+F^{(2)}(x)m^{(1)}m^{(2)}.
		\end{aligned}
	\end{equation}
	
	Now, with the discussion above at hand and Lemma $\ref{general heat ker}$, we are now in a position to present the proofs of Theorems $\ref{der F 2}$ and $\ref{der g2}.$
	
	\begin{proof}[Proof of Theorem $\ref{der F 2}$]
		Consider the following MFG systems for $j=1,2$:
		\begin{equation}\label{general H for F}
			\begin{cases}
				-\p_tu_j(x,t)-\Delta u_j(x,t)+ H(x,\nabla u_j)= F_j(x,m(x,t)),& \text{ in }\mathbb{T}^n\times (0,T),\medskip\\
				\p_t m_j(x,t)-\Delta m_j(x,t)-{\rm div} (m_j(x,t) H_p (x,\nabla u_j))=0,&\text{ in }\mathbb{T}^n\times (0,T),\medskip\\
				u_j(x,T)=G(x,m_T), & \text{ in } \mathbb{T}^n,\medskip\\
				m_j(x,0)=m_0(x), & \text{ in } \mathbb{T}^n.
			\end{cases}  	
		\end{equation}	
		Recall the higher order linearization method in Section $\ref{HLM}$.
		Let
		$$u_{j}^{(1)}:=\p_{\varepsilon_1}u_{j}|_{\varepsilon=0},$$
		$$m_{j}^{(1)}:=\p_{\varepsilon_1}m_{j}|_{\varepsilon=0}.$$
		By combining $\eqref{compute for H1}$, $\eqref{compute for H2}$ and $\eqref{H linear l=1 eg}$, we can deduce that 
		\begin{equation}\label{H linear l=1}
			\begin{cases}
				-\p_tu_j^{(1)}(x,t)-\Delta u^{(1)}_j(x,t)+ A^{(1)}(x)\cdot \nabla u_j= F^{(1)}_j(x)m_j^{(1)}(x,t),& \text{ in }\mathbb{T}^n\times (0,T),\medskip\\
				\p_t m^{(1)}_j(x,t)-\Delta m^{(1)}_j(x,t)-{\rm div} ( m_j^{(1)}(x,t) A^{(1)}(x))=0,&\text{ in }\mathbb{T}^n\times (0,T),\medskip\\
				u^{(1)}_j(x,T)=G^{(1)}(x)m^{(1)}_j(x,T), & \text{ in } \mathbb{T}^n,\medskip\\
				m^{(1)}_j(x,0)=f_1(x), & \text{ in } \mathbb{T}^n,\medskip\\
			\end{cases}  	
		\end{equation}
		where 
		\[
		A^{(1)}(x)=(H^{(1,0,0,...,0)}(x),H^{(0,1,0,...,0)}(x),...,H^{(0,0,...,1)}(x) ).
		\]

		We extend $f_l$ from $\mathbb{T}^n$ to $\mathbb{R}^n$ periodically, and still denote it by $f_l$. By Lemma $\ref{linearapp wellpose}$, 	$m_j^{(1)}$ is unique determined by $f_1(x)$. We use change of variables as well as a similar strategy in the proof of Lemma $\ref{general heat ker}$. 
		
		Suppose $F^{(1)}_1(x)\not\equiv F^{(2)}_1(x)$, then there is a open subset $U\subset\mathbb{T}^n$ such that $F_1^{(1)}(x)\neq F_1^{(2)}(x) $ in $U.$ Given $\epsilon>0$, there exists $f_1\in C_+^{2+\alpha}(\mathbb{T}^n)$ such that $\|f_l-\chi_U\|_{L^2(\mathbb{T}^n)}\leq\epsilon $, where $\chi_U$ is the characteristic function of $U$.
		Then the classical prior estimate implies that
		$$\|m_1^{(1)}(x,t)-\chi_{U\times(0,T)}\|_{L^2(Q) }\leq C\epsilon, $$
		for some constant $C>0$ 		
		
		This implies that
		\begin{equation}
			\int_{0}^{T}\int_{\mathbb{T}^n}K(x-y,T-s)(F^{(1)}_1(y)-F^{(1)}_2(y))\chi_U(y,s)\, dyds= 0.
		\end{equation}
		Since $K>0$ in $Q$ , it is a contradiction.
		Hence,
		$F_1^{(1)}(x)=F_2^{(1)}(x).$
		
		Next, we can consider the case $N=2.$ Let
		$$u_{j}^{(1,2)}:=\p_{\varepsilon_1}\p_{\varepsilon_2}u_{j}|_{\varepsilon=0},\quad
		m_{j}^{(1,2)}:=\p_{\varepsilon_1}\p_{\varepsilon_2}m_{j}|_{\varepsilon=0},$$
		and
		$$u_{j}^{(2)}:=\p_{\varepsilon_2}u_{j}|_{\varepsilon=0},\quad m_{j}^{(2)}:=\p_{\varepsilon_2}m_{j}|_{\varepsilon=0}.$$
		We can conduct the second-order linearization. Following a similar process as that in $\eqref{compute H 12 eg}$, we can deduce that
		\begin{equation}
			\begin{cases}
				-\p_tu_j^{(1,2)}-\Delta u^{(1,2)}_j+A^{(1)}\cdot\nabla u_j^{(1,2)}+R_1(x,t)\\
				\hspace*{3cm}= F_j^{(1)}(x)m_j^{(1,2)}+F^{(2)}_j(x)m_j^{(1)}m_j^{(2)},& \text{ in }\mathbb{T}^n\times (0,T),\medskip\\
				\p_t m^{(1,2)}_j(x,t)-\Delta m^{(1,2)}_j(x,t)-{\rm div} ( m_j^{(1)}(x,t) A^{(1)}(x))\\
				\hspace*{3cm}= R_2(x,t) ,&\text{ in }\mathbb{T}^n\times (0,T),\medskip\\
				u^{(1,2)}_j(x,T)=G^{(2)}(x)m^{(1,2)}_j(x,T), & \text{ in } \mathbb{T}^n,\medskip\\
				m^{(1,2)}_j(x,0)=0, & \text{ in } \mathbb{T}^n.
			\end{cases}  	
		\end{equation}
		where $$R_1(x,t)= \sum_{|\beta|=2}H^{(\beta)}(x)u_j^{(1)}u_j^{(2)}, $$ and 
		$$R_2(x,t)={\rm div}(m_j^{(1)} U^{(2)})+{\rm div}(m_j^{(2)} U^{(1)}).$$
		Here, the $l$-th component of $U^{(1)}$ is 
		$$U_l^{1}=\sum_{i=1}^{n}\frac{\p^2H}{\p z_l\p z_i}(x,0)\frac{\p u_j^{(2)}}{\p x_l},$$ 
		and the $l$-th component of $U^{(2)}$ is 
		$$U_l^{1}=\sum_{i=1}^{n}\frac{\p^2H}{\p z_l\p z_i}(x,0)\frac{\p u_j^{(1)}}{\p x_l}.$$

		Following a similar argument to the case $N=1$ (considering $m_0=\varepsilon_2f_2$ ), we have 
		$$u^{(1)}_1(x,t)= u^{(1)}_2(x,t),\quad  u^{(2)}_1(x,t)=u^{(2)}_2(x,t),$$
		and
		$$ m^{(1)}_1(x,t)= m^{(1)}_2(x,t),\quad  m^{(2)}_1(x,t)= m^{(2)}_2(x,t).$$
		By Lemma $\ref{linearapp wellpose}$, $m_j^{(1,2)}$ is unique determined by $f_1(x),f_2(x)$ and $G^{(1)}(x)$. By a similar argument, we readily have $F_1^{(2)}(x)=F_2^{(2)}(x).$
		
		Finally, by a mathematical induction, we can show the same result holds for $N\geq 3$. That is, for any $k\in\mathbb{N},$ we have $F^{(k)}_1(x)=F^{(k)}_2(x).$ Therefore, we have $F_1(x,z)=F_2(x,z).$
		
		The proof is complete. 
	\end{proof}
	
	We proceed with the proof of Theorem $\ref{der g2}$. To that end, we first state an auxiliary lemma, which is an analogue to Lemma $\ref{general heat ker}$, and omit its proof.
	
	\begin{lem}\label{general heat ker2}
		Let $g_1,g_2\in C^{2+\alpha}(\mathbb{T}^n)$ and $A(x)\in C^{2+\alpha}(\mathbb{T}^n)^n$. Consider the following systems with $f\in C^{2+\alpha,1+\frac{\alpha}{2}}(Q)$ and $j=1,2$:
		\begin{equation}
			\begin{cases}
				\p_tu_j(x,t)-\Delta u_j(x,t)+ A(x)\cdot \nabla u_j= f(x,t),& \text{ in } \mathbb{T}^n\times(0,T),\medskip\\
				u_j(x,0)=  g_j(x)v(x,T) , & \text{ in } \mathbb{T}^n .
			\end{cases} 
		\end{equation}
		Suppose for any $v\in  C^{2+\alpha,1+\frac{\alpha}{2}}(Q)$, we have $u_1(x,T;v)=u_2(x,T;v)$. Then it holds that $g_1(x)=g_2(x).$
	\end{lem}
	
	Next, we give the proof Theorem $\ref{der g2}.$
	
	\begin{proof}[Proof of Theorem $\ref{der g2} $]
		We shall follow a similar strategy that was developed for the proof of Theorem $\ref{der F 2}$. 
		Consider the following systems for $j=1,2$: 
		\begin{equation}\label{general H for g}
			\begin{cases}
				-\p_tu_j(x,t)-\Delta u_j(x,t)+ H(x,\nabla u_j)= F(x,m(x,t)),& \text{ in }\mathbb{T}^n\times (0,T),\medskip\\
				\p_t m_j(x,t)-\Delta m_j(x,t)-{\rm div} (m_j(x,t) H_p (x,\nabla u_j))=0,&\text{ in }\mathbb{T}^n\times (0,T),\medskip\\
				u_j(x,T)=G_j(x,m_T), & \text{ in } \mathbb{T}^n,\medskip\\
				m_j(x,0)=m_0(x), & \text{ in } \mathbb{T}^n.\\
			\end{cases}  	
		\end{equation}	
		We next perform the successive linearization process. Consider the case $N=1.$ Let
		$$u_{j}^{(1)}:=\p_{\varepsilon_1}u_{j}|_{\varepsilon=0},$$
		$$m_{j}^{(1)}:=\p_{\varepsilon_1}m_{j}|_{\varepsilon=0}.$$
		By direct computations, one can show that $(u_{j}^{(1)},v_{j}^{(1)} )$ satisfies the following system:
		\begin{equation}
			\begin{cases}
				-\p_tu_j^{(1)}(x,t)-\Delta u^{(1)}_j(x,t)+ A^{(1)}(x)\cdot \nabla u_j= F^{(1)}(x)m_j^{(1)}(x,t),& \text{ in }\mathbb{T}^n\times (0,T),\medskip\\
				\p_t m^{(1)}_j(x,t)-\Delta m^{(1)}_j(x,t)-{\rm div} ( m_j^{(1)}(x,t) A^{(1)}(x))=0,&\text{ in }\mathbb{T}^n\times (0,T),\medskip\\
				u^{(1)}_j(x,T)=G_j^{(1)}(x)m^{(1)}_j(x,T), & \text{ in } \mathbb{T}^n,\medskip\\
				m^{(1)}_j(x,0)=f_1(x), & \text{ in } \mathbb{T}^n.
			\end{cases}  	
		\end{equation}
		We can solve this system by first deriving $m^{(1)}_j$ and then obtaining $u^{(1)}_j.$ 
		
		Since  $\mathcal{M}_{G_1}=\mathcal{M}_{G_2}$, we have $$ u_1^{(1)}(x,0)=u_2^{(1)}(x,0),$$ for all $f_1\in C_+^{2+\alpha}(\mathbb{T}^n).$ By Lemma $\ref{general heat ker2}$, we readily see that $ G_1^{(1)}(x)=G_2^{(2)}(x).$
		
		Finally, by following a similar argument in the proof of Theorem~\ref{der F 2}, we can conduct the higher-order linearization process to show that $G_1^{(k)}(x)=G_2^{(k)}(x)$ for all $k\in\mathbb{N}$. Hence, $G_1(x,z)=G_2(x,z).$
		
		The proof is complete. 
	\end{proof}
	
	\section*{Acknowledgment}
	The work of H Liu was supported by Hong Kong RGC General Research Funds (project numbers, 11300821, 12301420 and 12302919) and the NSFC/RGC Joint Research Grant (project number, N\_CityU101/21).

\end{document}